\documentclass[11pt, letterpaper]{article}
\usepackage{amssymb}
\usepackage{amsmath}
\usepackage[top=1.2in, bottom=1.3in, left=1in, right=1in]{geometry}
\usepackage{parskip}
 \usepackage[T1]{fontenc}
\usepackage{lmodern}
 \usepackage{centernot}
\usepackage{epsfig}
\usepackage{amsmath,amsthm,amsfonts,amssymb,cite,amscd}
\usepackage{indentfirst}
\usepackage{mathrsfs}
\usepackage{amsthm}
\usepackage{tikz}
\usepackage{mathtools}
\usepackage{calc}
\usepackage[symbol]{footmisc}
\usetikzlibrary{patterns,arrows,decorations.pathreplacing}
\usetikzlibrary{fadings}
\usepackage{pgfplots}
\usepackage{parskip}
\usepackage{hyperref}
 \usepackage{booktabs}  

 \usepackage{bm} 
 \usepackage{tkz-euclide}

\usepackage{float}   

\newtheorem{theorem}{Theorem}[section]

\newtheorem{lemma}[theorem]{Lemma}
\newtheorem{corollary}[theorem]{Corollary}
\newtheorem{definition}[theorem]{Definition}
\newtheorem{claim}{Claim}
\newtheorem{conjecture}[theorem]{Conjecture}

\DeclareMathOperator{\ex}{ex}

\begin{document}
\baselineskip=0.175in

\begin{center}
{\LARGE  \textbf{Extremal graphs for the odd prism}\footnote{
E-mail addresses: hexc2018@qq.com (X. He),
ytli0921@hnu.edu.cn (Y. Li, corresponding author),
fenglh@163.com (L. Feng). }
}
\vskip 0.5cm
\end{center}

\begin{center}
Xiaocong He,
Yongtao Li$^*$,
Lihua Feng
\end{center}


\vspace{3mm}

\begin{center}
School of Mathematics and Statistics, HNP-LAMA, Central South University \\
 Changsha, Hunan, 410083, PR China  
\end{center}

\vspace{3mm}

\begin{center}
\today
\end{center}

\vspace{5mm}

\noindent
{\bf Abstract}:\ \
The Tur\'an number  $\mathrm{ex}(n,H)$ of a graph $H$ is the maximum number of edges in an $n$-vertex
graph  which does not contain $H$ as a subgraph.
The Tur\'{a}n number of regular polyhedrons was widely studied in a series of works due to Simonovits. In this paper, we shall present the
exact Tur\'{a}n number of the  prism $C_{2k+1}^{\square} $,
which is defined as the Cartesian product of
an odd cycle $C_{2k+1}$ and an edge $ K_2 $.
Applying a deep theorem of Simonovits
and a stability result of Yuan [European J. Combin. 104 (2022)], we shall
determine the exact value of  $\mathrm{ex}(n,C_{2k+1}^{\square})$ for every $k\ge 1$ and sufficiently large $n$, and we also characterize the extremal graphs.
Moreover,
in the case of $k=1$, motivated by a recent result of Xiao, Katona, Xiao and
Zamora [Discrete Appl. Math. 307 (2022)], we will determine the exact value of $\mathrm{ex}(n,C_{3}^{\square} )$ for every $n$ instead of for sufficiently large $n$. 

\vspace{3mm}
\noindent
{\bf Keywords}:\ \ Tur\'an number, Extremal graphs, Cartesian product.

\vspace{1mm} \noindent{\bf AMS classification}: 05C35.

\section{Introduction}
\noindent

In this paper, all graphs considered are undirected, finite and contain neither loops nor multiple edges.
The vertex set of a graph $G$ is denoted by $V(G)$, the edge set of $G$ by $E(G)$, and the number of edges in $G$ by $e(G)$.
 Let $K_n$ be the {\it complete graph} on $n$ vertices, and
  $K_{s,t}$ be the {\it complete bipartite graph} with parts of sizes
  $s$ and $t$.
 We write  $C_n$ and $P_n$ for the {\it cycle} and
 {\it path} on $n$ vertices, respectively.
The  {\it Tur\'{a}n graph} $T_r(n)$
is an $n$-vertex complete $r$-partite graph with each part of size $n_i$ such that
$|n_i-n_j|\le 1$ for all $1\le i,j\le r$, that is, $n_i$ equals
$\lfloor {n}/{r} \rfloor$  or $\lceil {n}/{r}\rceil$.
Sometimes, we  call $T_r(n)$ the {\it balanced complete $r$-partite graph}.
Denote by $G_1\cup G_2$ the vertex-disjoint union of two graphs
$G_1$ and $G_2$. For simplicity, we write $kG$
for the vertex-disjoint union of $k$ copies of $G$.
Denote by $G_1\otimes G_2$ the graph obtained from $G_1\cup G_2$ by joining each vertex of $G_1$
to each vertex of $G_2$.
For a vertex $v\in V(G)$, and a subgraph $H\subseteq G$ (possibly $v\in V(H)$), we use $N_H(v)$ to denote the set of neighbors of $v$ in $V(H)$, and use $d_H(v)$ to denote $|N_H(v)|$. For two sets $A,B\subseteq V(G)$,
we write $G[A,B]$ for the induced bipartite graph of $G$
with edges between $A$ and $B$, and $e(A,B)$ for the number of edges of $G[A,B]$.  We write $G[A]$ for the subgraph induced by $A$.
We denote the minimum degree of $G$ by $\delta(G)$,
the complement of $G$ by $\overline{G}$.
In particular, $\overline{K}_n$ is the empty graph,
which is also called an independent set on $n$ vertices.
 The chromatic number of $G$, denoted by $\chi(G)$,
 which is the minimum integer $s$ such that
 there exists a coloring of the vertex set $V(G)$ with $s$ colors
 and the adjacent vertices have different colors.

\subsection{Tur\'{a}n number of regular polyhedrons}

Let $G$ and $H$ be two simple graphs.
We say that $G$ is {\it $H$-free} if it does not contain
 a subgraph isomorphic to $H$.
For example, every bipartite graph is triangle-free.
The {\it Tur\'an number} of  $H$, denoted by $\ex(n,H)$, is the maximum number of edges in an $n$-vertex $H$-free graph.
Moreover, an $n$-vertex $H$-free graph with maximum number of
edges is called an {\it extremal graph} for $H$.
In general, it is possible that there are many (not only one) extremal graphs for $H$.
For instance, our result in this article
shares this phenomenon.
The problem of determining Tur\'{a}n numbers is one of the cornerstones of graph theory.
Observe that the  balanced complete $r$-partite graph
$T_r(n)$ is $K_{r+1}$-free.
In 1941, Tur\'{a}n \cite{Turan41} proved
$\mathrm{ex}(n,K_{r+1})=e(T_r(n))$.
Moreover, $T_r(n)$  is the unique $K_{r+1}$-free graph attaining the maximum number of edges.
In particular, for $r=2$, it gives
\begin{equation} \label{eq-man}
\mathrm{ex}(n,K_3)=\lfloor {n^2}/{4}\rfloor,
\end{equation}
 which was early proved by Mantel \cite{MAN};
see, e.g., the monograph \cite[p. 294]{Bo} for alternative proofs and extensions,
and the surveys \cite{FS13,Sim13} for more related results.

After Tur\'{a}n \cite{Turan41} determined the extremal graphs of the tetrahedron $K_4$,
he asked the following analogous problem:
Determining the Tur\'{a}n number of a graph consisting of the vertices and edges of a regular polyhedron. For example,
the cube $Q_8$, the octahedron $O_6$,
the dodecahedron $D_{20}$ and the icosahedron $I_{12}$;
see Figure \ref{regular-polyhedrons}.

\begin{figure}[H]
\centering
\includegraphics[scale=0.9]{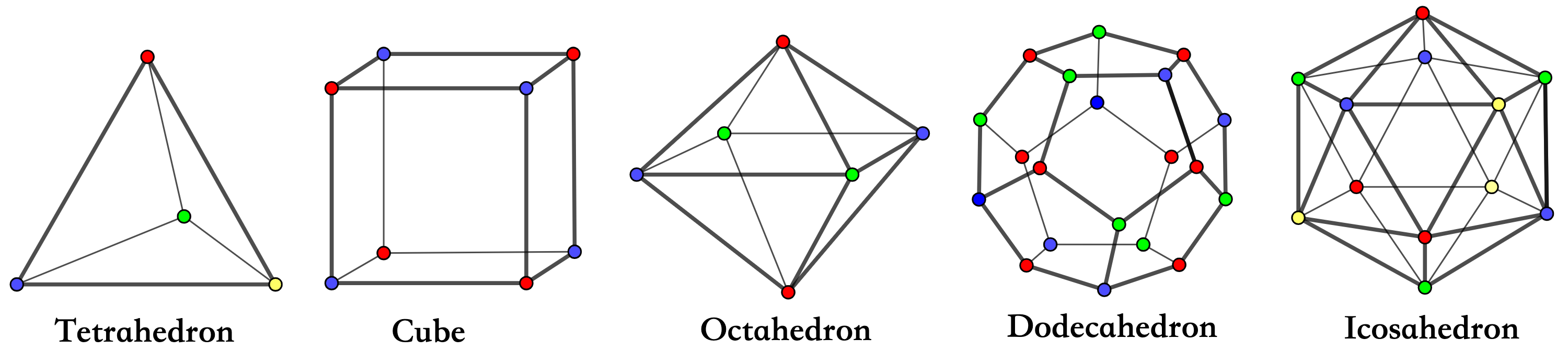}
\caption{The five regular polyhedrons}
\label{regular-polyhedrons}
\end{figure}

The Tur\'{a}n problem involving regular polyhedrons
was widely studied in the literature;
see, e.g., \cite{ES1970,PS2005,Fur2013,JN2017,JMY2021} for the cube $Q_8$,
\cite{ES1971} for the octahedron $O_6$,
\cite{Simonovits1974} for the dodecahedron $D_{20}$,
 and \cite{Sim1974} for the icosahedron $I_{12}$.
 In particular, although Erd\H{o}s and Simonovits \cite{ES1970} showed that $\mathrm{ex}(n,Q_8)\le Cn^{8/5}$ for some
 constant $C>0$,
 the lower bound on the Tur\'{a}n number of $Q_8$ remains open.
Motivated by the above results,
we shall consider the Tur\'{a}n problem of the prism, which can be viewed as
a variant graph of the cube $Q_8$.
Recall that $\chi (H)$ denotes the chromatic number of $H$.
The celebrated Erd\H{o}s--Stone--Simonovits
Theorem \cite{EDR1,EDR2} states that
\begin{align}\label{E-S-S}
\ex(n,H)= \left(1-\frac{1}{\chi(H)-1} \right)\binom{n}{2}+o(n^{2}).
\end{align}
This result determines the asymptotic value of $\ex(n,H)$ for every non-bipartite graph $H$, and it reduces to $\mathrm{ex}(n,H)=o(n^2)$
for every bipartite graph $H$.
At the first glance, the formula (\ref{E-S-S}) gives a satisfactory answer to the
problem of estimating $\mathrm{ex}(n,H)$.
Furthermore, it is an important problem in extremal
combinatorics  to obtain more accurate bound on the error term
$o(n^2)$.
In particular, one of the central problems  is
 to determine the exact Tur\'{a}n number for a graph $H$.

There are many elegant results on finding
the exact value of $\mathrm{ex}(n,H)$
for a specific graph $H$. For example, the earliest result due to Mantel  \cite{MAN} provided
$\ex(n,C_{3})= \lfloor n^2/4\rfloor$.
Erd\H{o}s, F\"{u}redi, Gould and Gunderson \cite{intersecting},
 Chen, Gould, Pfender and Wei \cite{Chen03}
determined the exact Tur\'{a}n number for intersecting triangles
and cliques, respectively.
Liu \cite{LIU} determined the Tur\'{a}n number
for the edge-blow-ups of cycles and a large class of trees,
and Ni, Kang, Shan and Zhu \cite{NKSZ2020} investigated
the edge-blow-ups of keyrings.
 In 2022, Yuan \cite{Yuan2022} studied the Tur\'{a}n number and
 extremal graphs for
the edge-blow-ups of complete graphs and complete bipartite graphs.
Applying the stability theorem of Erd\H{o}s and Simonovits,
 Hou, Qiu and Liu \cite{HQL16,HQL18},
 and Yuan \cite{Yuan2018} independently determined the Tur\'{a}n number for flower graphs.
Moreover, Dzido \cite{Dzi2013},
Dzido and Jastrzebski \cite{DJ2018}, and Yuan \cite{Yuan2021}
determined the Tur\'{a}n number for wheels.
In 2022, Xiao, Katona, Xiao and Zamora \cite{Xiao},
and Yuan \cite{YUAN} characterized the extremal graphs for the power of a path.

\subsection{The prism and Cartesian product}

Apart from the five regular polyhedrons in Figure \ref{regular-polyhedrons},
 there are two well-known
polyhedrons, called the pyramid and prism; see Figure \ref{odd-wheel-prism}. In fact, it is easy to observe that the pyramid is actually a wheel graph, which was studied by Dzido \cite{Dzi2013}, Dzido and Jastrzebski \cite{DJ2018} and Yuan \cite{Yuan2021}.
Differing from the study of the pyramid (wheel),
in this paper, we shall pay attention to the Tur\'{a}n number of the prism,
and characterize the extremal graphs.

\begin{figure}[H]
\centering
\includegraphics[scale=0.7]{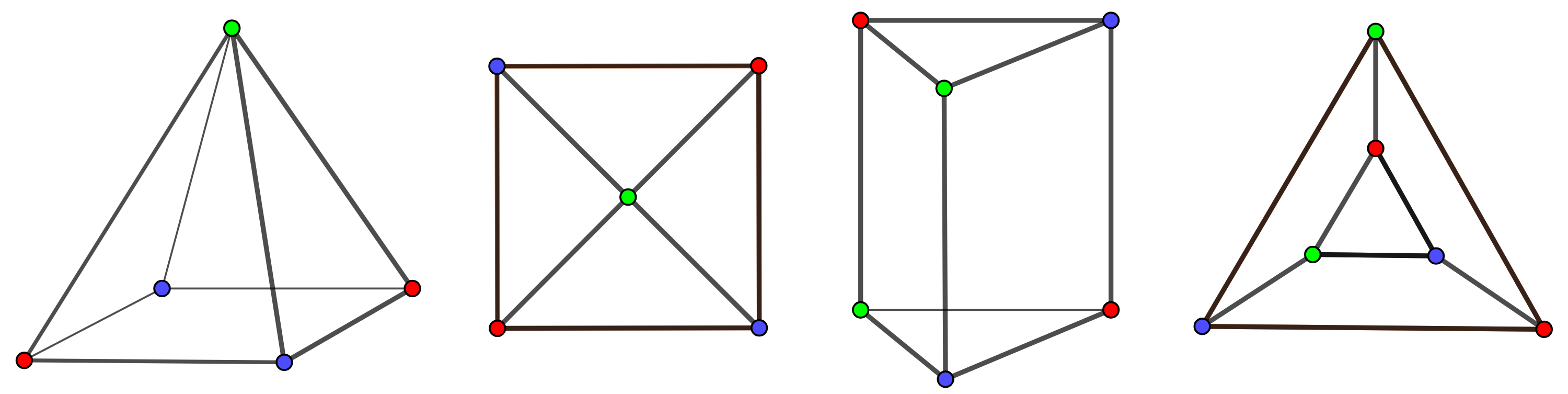}
\caption{The quadrangular pyramid and triangular prism}
\label{odd-wheel-prism}
\end{figure}

One can consider the prism  from the
perspective of Cartesian product.
The {\it Cartesian product} of graphs $G$ and $F$, denoted by $G\square F$, has vertex set  $V(G)\times V(F)$, in which two distinct vertices $(g_1,f_1)$ and $(g_2,f_2)$ are adjacent in $G\square F$ if either $g_1=g_2$ and $f_1f_2\in E(F)$, or
$f_1=f_2$ and $g_1g_2\in E(G)$.
In 2022, Brada\v{c}, Janzer, Sudakov and Tomon \cite{BJST2022}
studied the Tur\'{a}n number of the grid $P_t \square P_t$,
where $P_t$ is the path on $t$ vertices.
More precisely, they proved that for  $t\ge 2$,
there exist positive numbers $C_1$
and $C_2$ depending only on $t$ such that
$ C_1n^{3/2}\le \mathrm{ex}(n,P_t\square P_t) \le C_2n^{3/2}$.
In our paper, we shall consider the Tur\'{a}n number
for the Cartesian product of an odd cycle and an edge.
Recall that $C_{2k+1}$ is the cycle on $2k+1$ vertices.
Observe that the odd prism $C_{2k+1}^{\square} :=
C_{2k+1} \square K_2$
is the graph consisting of
two vertex disjoint cycles of order $2k+1$
and a matching pairing the corresponding vertices of these two cycles.
For example,  the third and fourth graph in Figure \ref{odd-wheel-prism}
are the triangular prism.
As promised, we shall study the Tur\'{a}n number of the prism
$C_{2k+1}^{\square}$.
By Erd\H{o}s--Stone--Simonovits' theorem in (\ref{E-S-S}),
we get an asymptotic value
\[ \ex(n,C_{2k+1}^{\square} )=\frac{n^{2}}{4}+o(n^{2}). \]
Furthermore,
one natural problem is to refine the error term and
 determine the exact Tur\'{a}n number of $C_{2k+1}^{\square}$.
Our result in this paper  solves this problem and
study the exact value of $\ex(n,C_{2k+1}^{\square})$. Moreover, we also characterize the extremal graphs for $C_{2k+1}^{\square} $.

\subsection{Main results}

The first  result in this paper can be stated as below.

\medskip
\begin{theorem}\label{main}
Let $k\ge 1$ be fixed and $n$ be sufficiently large. Then
$$\emph{ex}(n,C_{2k+1}^{\square} )=
\max\limits_{n_a+n_b=n} \left\{n_a(1+n_b)+\frac{1}{2}(j^2-3j):j\in \{0,1,2\},j \equiv n_a (\mathrm{mod}~{3})\right\}.$$
Moreover, all extremal graphs for $C_{2k+1}^{\square}$ are of the form of a complete bipartite graph $K_{n_a,n_b}$ with an extremal graph for $P_4$ added to the part of size $n_a$.
\end{theorem}

There are a few kinds of graphs $H$
for which $\mathrm{ex}(n,H)$ are known exactly {\it for all integers $n$},
including cliques, matchings, paths, odd cycles and some other special graphs. To be more specific,
Tur\'{a}n \cite{Turan41} determined the value of $\mathrm{ex}(n,K_{r+1})$
for all $n$,
 Erd\H{o}s and Gallai \cite{EG}
 determined $\mathrm{ex}(n,P_k)$ and
 $\mathrm{ex}(n,M_{2k})$,  where $M_{2k}$
 is the matching on $2k$ vertices.
  F\"{u}redi and Gunderson \cite{Furedi}
 determined $\mathrm{ex}(n,C_{2k+1})$ for all $n$ and $k\ge 2$.
 Bushaw and Kettle \cite{BK2011} determined
 $\mathrm{ex}(n,kP_3)$ for all $n\ge 7k$,
and later, Campos and Lopes \cite{CL2018},
Yuan and Zhang  \cite{YZ2017}, independently determined
the value $\mathrm{ex}(n,kP_3)$ for all integers $n$ and $k$.
Moreover, Bielak and Kieliszek
 \cite{BK2016} determined  $\mathrm{ex}(n,2P_5)$ for all $n$,
which was  also proved by Yuan and Zhang \cite{YZ2021} independently.
 In addition, Lan, Qin and Shi \cite{LQS2019}  determined  $\mathrm{ex}(n,2P_7)$
 for all $n$.

For the triangular prism in Figure \ref{odd-wheel-prism},
that is, the case $k=1$ in Theorem \ref{main},
we will use a different method, and
get rid of the condition that $n$ is sufficiently large.
More precisely, we will determine $\mathrm{ex}(n,C_3^{\square})$
for all $n$, and present all  $C_3^{\square}$-free graphs on
$n$ vertices with
maximum number of edges.
Our results will be presented in the following two theorems,
where we  provide  the Tur\'{a}n number first, and  give
the corresponding extremal graphs later.

\medskip
\begin{theorem}\label{C3}
The maximum number of edges in an $n$-vertex $C_3^{\square}$-free graph $(n\neq5)$ is
\[  \ex(n,C_3^{\square} )=\begin{cases}
\lfloor\frac{n^{2}}{4} \rfloor+ \lfloor\frac{n-1}{2} \rfloor, &\text{if}~n\equiv1,2,3 ~(\bmod~6); \\
\lfloor\frac{n^{2}}{4} \rfloor+ \lceil\frac{n}{2} \rceil, &\text{otherwise}.
\end{cases}
 \]
\end{theorem}

Before showing the extremal graphs in Theorem \ref{C3},
we need to define some graphs, which originally appears in Xiao, Katona, Xiao and Zamora \cite{Xiao}.
Let $K_{i,n-i}$ be the complete bipartite graph on parts $X$ and $Y$
with $|X|=i$ and $|Y|=n-i$, respectively.
On the one hand, if  $3 \mid i$, then embedding $\frac{i}{3}$ vertex-disjoint triangles in the class $X$ and keeping the class $Y$ unchanged, we get a new graph on $n$ vertices, and  denoted by $H^{i}_{n}$. On the other hand,
if $3\centernot\mid i$, then choosing an integer $ j\in [1, i]$ such that $3\mid (i-j)$,  we define $F^{i,j}_{n}$ as the graph obtained from $K_{i,n-i}$
by embedding a star on $j$ vertices, and
$\frac{i-j}{3}$ vertex-disjoint triangles  in the class $X$,
where  the star is disjoint from any triangle.
Clearly, one can observe that both $H_n^i$ and $F_n^{i,j}$
are $C_{2k+1}^{\square}$-free for any $i$ and $j$.
In addition, for small  $n$, we  introduce three
exceptional graphs $G_1,G_2$ and $G_3$; see Figure \ref{678}.

\begin{figure}[H]
\centering
\begin{tikzpicture}[scale=1.2]
\filldraw[thick] (0,0) circle (1.5pt) node[below] {$v_3$}
-- (-1,0.5) circle (1.5pt) node[below] {$v_1$}
 -- (0,1) circle (1.5pt) node[above] {$v_2$}
 -- (2,0.5) circle (1.5pt) node[below] {$v_5$}
 -- (1,1) circle (1.5pt) node[above] {$v_6$}
 -- (1,0) circle (1.5pt) node[below] {$v_4$} -- (0,0) circle (1.5pt) -- (0,1) circle (1.5pt) -- (1,1) circle (1.5pt) -- (0,0) circle (1.5pt) -- (2,0.5) circle (1.5pt) -- (1,0) circle (1.5pt) -- (0,1) circle (1.5pt);

\draw (0.5,-0.6) node{$G_1$};

\filldraw[thick] (3,1) circle (1.5pt) node[above] {$v_5$}
-- (4,0) circle (1.5pt) node[below] {$v_3$}
-- (3,0) circle (1.5pt) node[below] {$v_6$}
-- (3,1) circle (1.5pt) -- (4,1) circle (1.5pt) node[above] {$v_4$}
-- (6,0.5) circle (1.5pt) --
(5,1) circle (1.5pt) node[above] {$v_1$}--
(5,0) circle (1.5pt) node[below] {$v_2$}
-- (4,0) circle (1.5pt) -- (4,1) circle (1.5pt) -- (5,1) circle (1.5pt) -- (4,0) circle (1.5pt) -- (6,0.5) circle (1.5pt)
node[below] {$v$}-- (5,0) circle (1.5pt) -- (4,1) circle (1.5pt);

\filldraw[thick] (3,0) circle (1.5pt) -- (4,1) circle (1.5pt);

\draw (4.5,-0.6) node{$G_2$};

\filldraw[thick] (8,1) circle (1.5pt) node[above] {$v_5$}
-- (9,0) circle (1.5pt) node[below] {$v_3$}
-- (8,0) circle (1.5pt) node[below] {$v_6$}
-- (8,1) circle (1.5pt) -- (9,1) circle (1.5pt) node[above] {$v_4$}
-- (11,0.5) circle (1.5pt) node[below] {$x$}
-- (10,1) circle (1.5pt) node[above] {$v_1$}
 -- (10,0) circle (1.5pt) node[below] {$v_2$}
 -- (9,0) circle (1.5pt) -- (9,1) circle (1.5pt) -- (10,1) circle (1.5pt) -- (9,0) circle (1.5pt) -- (11,0.5) circle (1.5pt) -- (10,0) circle (1.5pt) -- (9,1) circle (1.5pt)
 (8,0) circle (1.5pt) -- (9,1) circle (1.5pt)
  (8,0) circle (1.5pt) -- (7,0.5) circle (1.5pt) node[below] {$y$}
  -- (8,1) circle (1.5pt)
   (9,0) circle (1.5pt) -- (7,0.5) circle (1.5pt) -- (9,1) circle (1.5pt);

\draw (9,-0.6) node{$G_3$};
\end{tikzpicture}
\caption{The graphs $G_1,G_2$ and $G_3$.}
\label{678}
\end{figure}
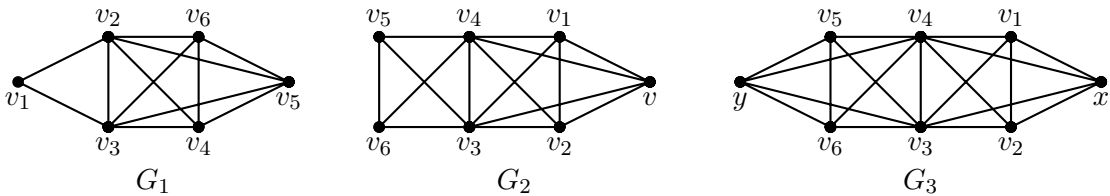

The extremal graphs for $C_3^{ \square}$ are presented
in the following tables.

\begin{theorem}\label{EC3}
For $n\geq 6$, the extremal graphs for $C_3^{\square}$ are the following ones.

\begin{table}[H]
\centering
\begin{tabular}{ccccccccccc}
\toprule
 order $n$ & 6 &  7 & 8  \\
\midrule
 extremal graphs & $G_1,H_6^3$~ &
 ~$G_2,F_7^{4,1}\!,F_7^{4,4}\!,H_7^3$~ &
 ~$G_3,F_8^{4,1}\!,F_8^{4,4}\!,F_8^{5,2}\!,F_8^{5,5}$ \\
\bottomrule
\end{tabular}
 \label{tab-n678}
\end{table}
\vspace{-4mm}
\noindent
Moreover, for $n\ge 9$, we have the following extremal graphs.
\begin{table}[H]
\centering
\begin{tabular}{ccccccc}
\toprule
 \!\!order $n~(\mathrm{mod}~6)$ & 0 & 1 & 2 & 3  & 4 & 5 \\
\midrule
 \!\!extremal graphs &
 $H_n^{{\frac{n}{2}}}$  &
 $F_n^{\lceil \!\frac{n}{2}\!\rceil, j} \!, H_n^{\lfloor \!{\frac{n}{2}}\!\rfloor}$ &
 $F_n^{{\frac{n}{2}}, j}\!,F_n^{{\frac{n}{2}}+1, j}$ &
 $F_n^{\lceil \!\frac{n}{2}\!\rceil, j}\!, H_n^{\lceil \!{\frac{n}{2}}\!\rceil+1}$
  &
 $H_n^{{\frac{n}{2}} +1}$
 & $H_n^{\lceil \!{\frac{n}{2}}\!\rceil}$ \\
\bottomrule
\end{tabular}
 \label{tab-n9}
\end{table}
\vspace{-4mm}
\noindent
where $j$ can be chosen as any possible integer satisfying  $1\le j\leq i$ and $3 \mid (i-j)$.
\end{theorem}

\medskip
\noindent
{\bf Outline of the paper.}
This paper is organized as follows.
In Section \ref{sec2},  we shall present some preliminaries,
including the Tur\'{a}n-type results on paths, and
a  deep structure theorem of Simonovits
for graphs whose decomposition family contains a linear forest.
In Section \ref{sec3}, we will give the proof of Theorem \ref{main}.
In Section \ref{sec4},
in order to determine the
$n$-vertex extremal graphs for $C_3^{ \square}$ for every $n\ge 6$, we will provide some lemmas for our purpose.
Later, the proofs of Theorems \ref{C3} and \ref{EC3}
are given in Section \ref{sec5}, respectively.
Finally, some open problems are proposed in Section \ref{sec6}.

\section{Preliminaries}

\label{sec2}

 \subsection{The Tur\'an-type results for paths}

To begin with, we need to introduce a result of Erd\H{o}s and Gallai \cite{EG}.

\medskip

\begin{theorem}[Erd\H os and Gallai \cite{EG}]
\label{Erd\H os and Gallai}
The maximum number of edges in an $n$-vertex $P_k$-free graph is ${n(k-2)\over 2}$, that is, $\ex (n,P_k)\leq {n(k-2)\over 2}$, with equality if and only if $(k-1)| n$ and the graph is a vertex disjoint union of ${n\over k-1}$ copies of $K_{k-1}$.
\end{theorem}

Note that the bound in Theorem \ref{Erd\H os and Gallai}
is sharp only in the case $(k-1)| n$. Furthermore,
for the general cases, say $n=(k-1)t +r$,
this bound was refined by Faudree and Schelp \cite{FS},
and independently by Kopylov \cite{K} using a different method.

\medskip
\begin{theorem}[Faudree and Schelp \cite{FS}, Kopylov \cite{K}] \label{extremal for path}
 Let $n=(k-1)t+r$ with $t\geq 1$ and $0\leq r\leq k-2$. Then
$$\ex(n,P_k)=t\binom{k-1}{2}+\binom{r}{2}.$$
Moreover, each $n$-vertex extremal graph for $P_k$ is isomorphic to either
$$ tK_{k-1}\cup K_{r} $$
or
$$ (t-s-1)K_{k-1}\cup\left(K_{\frac{k-2}{2}}\otimes\overline{K}_{\frac{k}{2}+s(k-1)+r}\right)$$ for some $s\in[0,t-1]$ when $k$ is even and $r\in\{\frac{k}{2}, \frac{k-2}{2}\}$.
\end{theorem}

In this paper, we shall use the case $k=4$.

\begin{corollary} \label{cor-6}
Let $G$ be a $P_4$-free graph on $n$ vertices.
Then
\[  e(G) \le n + \frac{1}{2}(j^2-3j), \]
where $j\in \{0,1,2\}$ and $j \equiv n (\mathrm{mod}~3)$. More precisely,
\begin{itemize}
\item[(a)]
 if $n=3t$, then $e(G)\le 3t$, with equality  if and only if
$G=tC_3$;

\item[(b)]
 if $n=3t+1$, then $e(G)\le 3t$,
with equality if and only if
\[ G \in \{ tC_3 \cup K_1 \} \cup \{
(t-\ell -1) C_3 \cup K_{1,3\ell +3}:  0\le \ell \le t-1 \}\];

\item[(c)]
 if $n=3t+2$, then $e(G)\le 3t+1$, with equality if and only if
\[ G\in \{ tC_3 \cup K_2\} \cup \{
(t-\ell -1) C_3 \cup K_{1,3\ell +4}:  0 \le \ell \le t-1 \}. \]
\end{itemize}
\end{corollary}

Recently, the stability result for (connected) $P_k$-free graphs
was successively studied by F\"{u}redi, Kostochka, Luo and Verstra\"{e}te \cite{FKV2016,FKLV2018}, and Ma and Ning \cite{MN2020}.
To prove Theorem \ref{main},
we need to apply the  following variant due to Yuan \cite{YUAN}.
In the original statement, it requires $t\ge 3$.
We mention here that
the result also holds for $t=2$ by a similar argument.

\begin{lemma}[Yuan \cite{YUAN}]
\label{stablity for non-connecte}
Let $c$ be a positive integer and $n$ be sufficiently large. Given an $n$-vertex $P_k$-free graph $G$ with $e(G)\geq \emph{ex}(n,P_k)-c$. Let $\ell=\min\{n/(16c),n/(2k-2)\}$ and $t=\lfloor {k}/{2} \rfloor \ge 2$. If $k$ is odd, then $G$ contains a copy of $\ell K_{t-1}$. If $k$ is even, then $G$ contains either a copy of $\ell K_{k-1}$ or a copy of $K_{t-1,\ell}$.
\end{lemma}

\subsection{A  deep theorem of Simonovits}

The decomposition family was firstly introduced by Simonovits \cite{Simonovits1982},
and it was used to study the Tur\'{a}n numbers
of non-bipartite graphs.
For a fixed graph $H$, the error term $o(n^2)$ in (\ref{E-S-S})
can be refined by the decomposition family of $H$.
Moreover, the decomposition family of a non-bipartite graph can help us to determine the finer structure of its extremal graphs.
Now let us give the definition of the decomposition family of a graph.

\medskip

\begin{definition}[Simonovits] \label{def-8}
Given a graph $L$ with  $\chi(L)=p+1\geq2$. Let $\mathcal{M}(L)$ be the family of minimal graphs $M$ satisfying the following:
there exists an integer $t\ge 1$ depending on $L$ such that
$L \subseteq (M\cup \overline{K}_t) \otimes T_{p-1}((p-1)t)$.
 We call $\mathcal{M}(L)$ the {\it decomposition family} of $L$.
\end{definition}

In other words, a graph $M$ belongs to $\mathcal{M}(L)$ if
 putting a copy of $M$ (but not any of its proper subgraphs) into one vertex part of a large $p$-partite Tur\'{a}n graph
$T_p(pt)$ will lead to a copy of $L$.
For example, we have $\mathcal{M}(sK_{r+1}) =\{sK_2\}$ for each
$s\ge 1$ and $r\ge 2$.
Moreover, for the dodecahedron $D_{20}$
 and icosahedron $I_{12}$ in Figure \ref{regular-polyhedrons},
 it was proved in \cite{Simonovits1974,Sim1974} that
$\mathcal{M}(D_{20})=\{6K_2\}$ and
$\mathcal{M}(I_{12})=\{P_6\}$, respectively.
In addition, for the wheels (the pyramid) in Figure \ref{odd-wheel-prism}, it can be seen in \cite{Yuan2021} that
$\mathcal{M}(W_{2k+1})=\{K_{1,k},C_{2k}\}$
and $\mathcal{M}(W_{2k})=\{K_2\}$,
since $\chi (W_{2k})=4$ and $\chi (W_{2k}-e)=3$
for each edge $e$ in the cycle.

\medskip
The following lemma follows immediately from the definition.

\medskip
\begin{lemma} \label{Lfree}
Given a graph $L$ with  $\chi(L)=p+1\geq2$.
If $G$ is a graph obtained from the $p$-partite Tur\'{a}n graph $T_p(n)$ by embedding an extremal graph for $\mathcal{M}(L)$ into one part, then $G$ is $L$-free. Consequently,
we get $\mathrm{ex}(n,L)\ge e(T_p(n)) + \mathrm{ex}(\lceil {n}/{p}\rceil, \mathcal{M}(L))$.
\end{lemma}

In the proof of Theorem \ref{main},
we will use the following theorem, which
attributes to a deep theorem of Simonovits \cite[p. 371]{Simonovits1974}.
Roughly speaking, it states that if the decomposition family
$\mathcal{M}(L)$  contains a linear forest
(a graph in which each component is a path),
then the extremal graphs for $L$ have a very simple and symmetric
structure.

\medskip
\begin{theorem}[Simonovits \cite{Simonovits1974}]
\label{family-extremal-graphs}
Let $L$ be a given graph with $p=\chi(L)-1$.
If the decomposition family $\mathcal{M}(L)$ contains a linear forest, then there exist $r=r(L)$ and $n_0=n_0(r)$ such that if $n\geq n_0$, then each $n$-vertex extremal graph $G$ for $L$ belongs to the family of graphs satisfying the following.
\begin{itemize}
\item[(a)]
 After deleting at most $r$ vertices of $G$,
the remaining graph $G'$ satisfies
 \[ G^\prime= G^1\otimes G^2 \otimes \cdots \otimes G^p, \]
  where the graph $G^i$ satisfies
  $||V(G^i)|-n/p|\leq r$ for each $1\leq i \leq p$.

  \item[(b)]
   For each $1\leq i\leq p$, the graph $G^i$ consists of vertex disjoint copies of non-isomorphic connected graphs $H^j_is$ with $|V(H_{i}^j)|\leq r$ such that any two copies, say $F_1$ and $F_2$, of $H_{i}^j$s are symmetric subgraphs in $G$: there exists an isomorphism $\psi:F_1\rightarrow F_2$ such that for every $x\in V(F_1)$ and $y\in G-F_1-F_2$, $xy$ is an edge if and only if $\psi(x)y$ is an edge.
   \end{itemize}
\end{theorem}

Theorem \ref{family-extremal-graphs} can benefit us to characterize
the approximate structure of all  extremal graphs for
the prism $C_{2k+1}^{\square}$.
We remark that this theorem was also used by Yuan \cite{YUAN}
for treating the Tur\'{a}n problem of the power of a path.
In particular, we refer the readers to \cite{Yuan2022}
for graphs $L$ whose decomposition family $\mathcal{M}(L)$ contains a matching.

\section{Proof of Theorem \ref{main}}

\label{sec3}

The following lemma is simple, but it plays an important role in our proof.

\medskip

\begin{lemma}\label{decompositionofpowerpath}
The decomposition family $\mathcal{M}(C_{2k+1}^{\square} )=\{P_4\}$.
\end{lemma}

\begin{proof}
The chromatic number $\chi(C_{2k+1}^{\square})=3$. For sufficiently large $t$, it is easy to see that embedding a $P_4$ into one partite of $K_{t,t}$ will lead to a copy of $C_{2k+1}^{\square}$, and 
embedding a proper subgraph of $P_4$ can not produce 
 such a copy. Hence, $P_4\in\mathcal{M}(C_{2k+1}^{\square})$. 
 We claim that $\mathcal{M}(C_{2k+1}^{\square})=\{P_4\}$. 
Suppose on the contrary that $\mathcal{M}(C_{2k+1}^{\square})\neq\{P_4\}$. 
Let $H\in\mathcal{M}(C_{2k+1}^{\square})$ be a graph different from 
$P_4$ and $H$ does not contain a subgraph isomorphic to $P_4$. 
Note that $C_{2k+1}^{\square}$ is  $3$-regular. 
Then $H$ is the disjoint union of some $K_2$, $K_{1,2}$, $K_{1,3}$ and triangles. For sufficiently large $t$, let $V_1$ and $V_2$ be two partite sets of $K_{t,t}$. Embedding a copy of $H$ into $V_1$ will lead to a $C_{2k+1}^{\square}$. We denote the new graph by $K_{t,t}+H$. Since $V_2$ is an independent set of $K_{t,t}+H$, we can see that $H$ has no triangle components. Observe that there are exactly two copies of $C_{2k+1}$ in $C_{2k+1}^{\square}$ and $V_2$ is an independent set. 
Then $H$ does not contain $K_2$ as a component. Since $C_{2k+1}$ is an odd cycle, there must be a $K_{1,2}$ component in $H$, say $uvw$,  such that $uv$ is an edge of one $C_{2k+1}$ and $vw$ is an edge between two $C_{2k+1}$ in $C_{2k+1}^{\square}$. Let $x\notin\{u,w\}$ be a neighbor of $v$ and $y,z\neq v$ be two neighbors of $w$. Then $xy$ or $xz$ is an edge of $C_{2k+1}^{\square}$. Observe that $\{x,y,z\}\subseteq V_2$ and $V_2$ is an independent set. Neither $xy$ nor $xz$ is an edge of $K_{t,t}+H$, a contradiction. Thus, such $H$ does not exist and $\mathcal{M}(C_{2k+1}^{\square})=\{P_4\}$.
\end{proof}

\begin{lemma}\label{no-two-edge}
Let $G=G^1\otimes G^2$ with $|V(G^i)|=n_i\geq 2k+1$ for $i=1,2$.
\begin{itemize}
\item[(a)]
Let $G'$ be the graph obtained from $G$ by adding a new vertex $y$ with $|N_{G^i}(y)|\geq3$ for $i=1,2$. If $G^1\cup G^2$ has an edge incident with $N_G(y)$, then $G'$ contains a copy of $C_{2k+1}^{\square} $.

\item[(b)]
If both $G^1$ and $G^2$  contain an edge, then $G$ contains a copy of $C_{2k+1}^{\square} $.
\end{itemize}
\end{lemma}

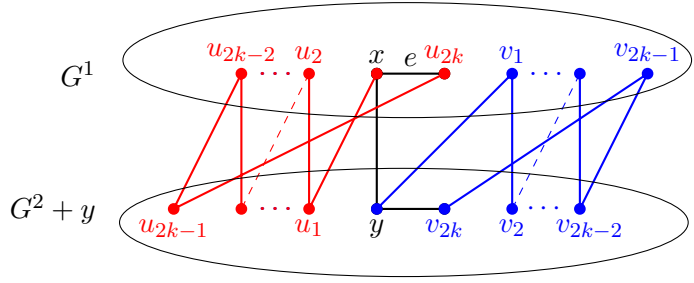
\begin{figure}[H]
\centering
\begin{tikzpicture}[scale=0.9]
\filldraw[black,thick] (0,0) circle (2pt) node[above]{$x$} -- (1,0) circle (2pt);
\filldraw[black,thick] (0,0) circle (2pt) -- (0,-2) circle (2pt) node [below]{$y$} -- (1,-2) circle (2pt);
\filldraw[blue,thick] (0,-2) circle (2pt) -- (2,0) circle (2pt) node[above]{$v_1$} -- (2,-2) circle (2pt) node[below]{$v_2$};
\draw[dashed,blue] (2,-2) -- (3,0);
\filldraw[blue,thick] (3,0) circle (2pt) -- (3,-2) circle (2pt) node[below]{$~~v_{2k-2}$} -- (4,0) circle (2pt) node[above]{$v_{2k-1}$} -- (1,-2) circle (2pt) node[below]{$v_{2k}$};
\filldraw[red,thick] (0,0) circle (2pt) -- (-1,-2) circle (2pt) node[below]{$u_1$} -- (-1,0) circle (2pt) node[above]{$u_2$};

\draw[dashed,red] (-1,0) -- (-2,-2);
\filldraw[red,thick] (-2,-2) circle (2pt) -- (-2,0) circle (2pt) node[above]{$u_{2k-2}$} -- (-3,-2) circle (2pt) node[below]{$u_{2k-1}$} -- (1,0) circle (2pt) node[above]{$u_{2k}$};

\draw[blue] (2.5,0) node{$\ldots$} (2.5,-2) node{$\cdots$} (-1.5,0) node{$\ldots$} (-1.5,-2) node{$\ldots$};
\draw[red] (-1.5,0) node{$\ldots$} (-1.5,-2) node{$\ldots$};
\draw (0.5,0.2) node{$e$} (-4,0) node[left]{$G^1$} (-4,-2) node[left]{$G^2+y$};

\draw (4.65,0.2) arc (0:360:4.2cm and 0.85cm)
 (4.6,-2.2) arc (0:360:4.2cm and 0.8cm);
\end{tikzpicture}
\caption{The graph $G'$.}
\label{G'}
\end{figure}

\begin{proof}
(a) Without loss of generality, we may assume that $G^1$ contains an edge $e$, which is incident with a vertex $x\in N_G(y)$; see Figure \ref{G'}.
Since $|N_{G^1}(y)|\ge 3$,
let $v_1\in V(G^1)$ be a neighbor of $y$ such that $v_1$ is not an end vertex of $e$. Let $v_{2k}\in V(G^2)$ be a vertex adjacent to $y$. Then $yv_1v_2\ldots v_{2k}$ and $xu_1u_2\ldots u_{2k}$ are two copies of $C_{2k+1}$ in $G'$; see
the  blue and red cycles in Figure \ref{G'}, respectively.
Observe that  the above two cycles, together with the matching $\{xy,u_1v_1,u_2v_2,\ldots,u_{2k}v_{2k}\}$, form a copy of $C_{2k+1}^{\square} $. Thus, we get $C_{2k+1}^{\square} \subseteq G'$, as required.

(b) We may assume that $e_1\in E(G^1)$  and $e_2\in E(G^2)$.
Similar with the case (a), one can find a red cycle $C_{2k+1}$
containing the edge $e_1$, and a blue cycle containing $e_2$.
Indeed,  the odd prism $C_{2k+1}^{\square} $ is contained in
$K_{2k+1,2k+1}+\{e_1,e_2\}$.
Combining with $(K_{2k+1,2k+1}+\{e_1,e_2\})\subseteq G$,
we then obtain a copy of $C_{2k+1}^{\square} $ in $G$.
 The proof is completed.
\end{proof}

The following well-known lemma  proved by Erd\H{o}s \cite{Erdos} and Simonovits \cite{Simonoivts1968} is a powerful tool in extremal graph problems; see, e.g.,  Corollary~4.3 in Chapter 6  of \cite{Bo}.

\medskip
\begin{lemma}[Erd\H{o}s \cite{Erdos} and Simonovits \cite{Simonoivts1968}]
\label{minmum degree}
Let $H$ be a graph with $\chi(H)=p+1\geq 3$. If $G$ is an extremal graph for $H$ on $n$ vertices, then $\delta(G)
= (1- {1}/{p})n+o(n)$.
\end{lemma}

Now, we are ready to give a proof  of  Theorem \ref{main}.

\begin{proof}[{\bf Proof of Theorem \ref{main}}]
Let $G$ be a  $C_{2k+1}^{\square} $-free graph on $n$ vertices
with maximum number of edges. Our goal is to prove that
$G$ is obtained from a complete bipartite graph $K_{n_a,n_b}$
(where $n_a+n_b=n$)
by embedding an $n_a$-vertex extremal graph for $P_4$ into the part of size $n_a$.
First of all, it follows from Lemmas \ref{Lfree} and \ref{decompositionofpowerpath} that
\begin{equation} \label{the lower bound}
\begin{aligned}
e(G)
&\geq \max\left\{n_an_b + \mbox{ex}(n_a,P_4):n_a+n_b=n\right\}.
\end{aligned}
\end{equation}
A simple calculation shows that
 \begin{equation*}\label{the lower bound 1}
e(G)\geq e(T_2(n)) +\frac{n}{2}+O(1).
\end{equation*}
Now, we are going to apply Theorem~\ref{family-extremal-graphs}.
After deleting at most $r$ vertices of $G$, the resulting graph
$G'$ can be written as the product $G^\prime=G^1\otimes G^2$ with $n/2-r\leq|V(G^i)|\leq n/2+r$ for each $i\in \{1,2\}$. Moreover, each component of $G^i$ is on at most $r$ vertices.
We denote by $D$ the set of deleted vertices.
Then $|D|\le r$. By Lemma~\ref{no-two-edge} (b), it is impossible that
both $G^1$ and $G^2$ contain an edge. Then
we may assume $G^2$ is an independent set of size $n/2+o(n)$. By Theorem \ref{family-extremal-graphs},
for each $i=1,2$, the isolated vertices of $G^i$ are symmetric subgraphs of $G$. Thus, for any vertex $z\in D$,
by the symmetry in Theorem~\ref{family-extremal-graphs},
the isolated vertices of $G^i$ either all are adjacent to $z$, or all are not adjacent to  $z$.

We divide $D$ into the following two sets: if $x\in D$ is not joint to vertices of $G^2$, then let it belongs to $D_2$, otherwise, let it belongs to $D_1$. Let $A_i=V(G^i)\cup D_i$ and denote $|A_i|=n_i$ for $i=1,2$.
Clearly, we have $V(G)=A_1\cup A_2$.
By the symmetry of Theorem \ref{family-extremal-graphs},
{\it each vertex of $A_1$ is joint to all vertices of $G^2$.}
 By Lemma~\ref{minmum degree}, we get $\delta (G) \ge n/2 + o(n)$.
Note  that $G^2$  is an empty graph and $|D_2|\le r$.
Consequently, each vertex of $G^2$ is adjacent to $n/2+o(n)$ vertices of $A_1$. Moreover, each vertex of $D_2$ is adjacent to $n/2 - o(n)$ vertices of $G^1$.

\medskip

\begin{claim}\label{claim}
$G[A_i]$ does not contain a copy of $P_4$ for each $i\in \{1,2\}$.
\end{claim}

\begin{proof}[Proof of Claim 1]
From the above discussion,
we know that the edges between $A_1$ and
$V(G^2)$ form a complete bipartite subgraph.
By Definition \ref{def-8} and
 Lemma  \ref{decompositionofpowerpath}, we have $G[A_1]$ is $P_4$-free.

Assume on the contrary that $G[A_2]$ contains a copy of $P_4$.
Since $A_2=V(G^2) \cup D_2$ and $E(G^2)$ is empty,
 the above copy of $P_4$ must be contained in $G[D_2]$.
Choosing a  set $U$ consisting of $4k$ vertices of $A_2\backslash V(P_4)$ arbitrarily,
by Lemma \ref{minmum degree},
we know that these $4k$ vertices and the vertices of $V(P_4)$
have at least $n/2-o(n)$ common neighbors in $G^1$.
By Definition \ref{def-8} and Lemma \ref{decompositionofpowerpath},
there is a copy of $C_{2k+1}^{\square}$ in the subgraph of $G$ induced by the set $U\cup V(P_4)$ together with its common neighborhood, and so
$G$ contains a copy of $C_{2k+1}^{\square} $, a contradiction. The proof is completed.
\end{proof}

Since $|D|\leq r$, by (\ref{the lower bound}), we have $e(G[A_1])\geq \mbox{ex}(|A_1|,P_4)-O(1)$. Thus, by Lemma~\ref{stablity for non-connecte}, $G[A_1]$ contains a copy of either $K_{1,\ell}$ or $\ell K_{3}$, where $\ell=\Theta(n)$.
Next, we divide the proof into two cases.
 For convenience, we write $N(U) : =\bigcap_{u\in U}N(u)$ for
the common neighborhood of vertices of a set $U$ and
denote $N_A(U):=N(U)\cap A$.

\medskip

{\bf Case 1.} $G[A_1]$ contains a copy of $\ell K_3$.
We claim that $A_2$ is an independent set in $G$.
Assume on the contrary that $G[A_2]$ contains an edge $xy$. Then $xy$ is contained in $G[D_2]$
since $V(G^2)$ is an independent set and there are no edges between $D_2$ and $V(G^2)$.
Let $U_2$ be a vertex set consisting of $x,y$ and
 any $2k-1$ vertices of $A_2\backslash\{x,y\}$.
By Lemma \ref{minmum degree},
these $2k+1$ vertices of $U_2$ have at least $|A_1|-o(n)$
common neighbors in $A_1$.
In other words, there are $o(n)$ vertices of $A_1$
which are outside of $N_{A_1}(U_2 )$.
Since $G[A_1]$ has $\ell $ disjoint copies of $K_3$ and
$\ell=\Theta(n)$, we obtain that there exists a triangle located in $ N_{A_1}(U_2) $.
Consequently, we can choose a set $U_1$ consisting of $2k+1$ vertices in
$N_{A_1}(U_2)$ such that $G[U_1]$ contain a copy of $K_3$.
Note  that $G[U_1]$ contains an edge
and $G[ U_2,U_1]$ forms a complete bipartite subgraph.
By Lemma \ref{no-two-edge} (b), we obtain that $G$ contains a copy of $C_{2k+1}^{\square} $, a contradiction. Thus, there is no edge in $G[A_2]$.
Denote  $|A_1|=n_1$ and $|A_2|=n_2$.
Therefore, it follows from Claim \ref{claim} that
\[ e(G) = n_1n_2 + e(A_1) \le \max\left\{n_an_b + \mbox{ex}(n_a,P_4):n_a+n_b=n\right\}.  \]
Combining with (\ref{the lower bound}),
we obtain the desired bound on $e(G)$.
Moreover, the extremal graphs for the prism $C_{2k+1}^{\square} $
are obtained from a complete bipartite graph $K_{n_a,n_b}$
by embedding an $n_a$-vertex extremal graph for $P_4$ in
the part of size $n_a$. Furthermore,
the extremal graphs for $P_4$ are
also characterized by applying Corollary \ref{cor-6}.

\medskip

{\bf Case 2.} $G[A_1]$ contains a copy of $K_{1,\ell}$.
Let $u$ be the central vertex of $K_{1,\ell}$. Then $u$ has at least $\ell=\Theta(n)$ neighbors in $A_1$. Recall that each vertex of $A_1$
is joint to all vertices of $G^2$.
We denote $\ell_i=|N_{A_i}(u)|$ for each $i\in \{1,2\}$.
 Then $\ell_i = \Theta (n)$.

\begin{claim} \label{claim2-no-edge}
There is no edge in $G[A_1 \setminus \{ u\}]$ incident with a vertex of $N_{A_1}(u)$.
 Correspondingly, there is no edge in $G[A_2]$ incident with a vertex of $N_{A_2}(u)$.
\end{claim}

\begin{proof}
  If there is an edge in $G[A_1\setminus \{u\}]$ incident with a vertex of $N_{A_1}(u)$, then $G[A_1]$ contains a copy of $P_4$,
  which contradicts with Claim 1.
Hence there is no edge in $G[A_1 \setminus \{ u\}]$ incident with a vertex of $N_{A_1}(u)$.
If there is an edge $xy$ in $G[A_2]$ incident with a vertex of $N_{A_2}(u)$, then $x,y\in D_2$ since $G^2$ has no edge and there are no edges between $D_2$ and $V(G^2)$.
Let $E_2$ be a vertex set consisting of $x,y$ and any
 $2k$ vertices of $N_{G^2}(u)$.
It follows that these $2k+2$ vertices of $E_2$
have $|A_1|-o(n)$ common neighbors in $A_1$.
 Since $u$ has  $\ell_1=\Theta(n)$ neighbors in $A_1$,
 we know that
 \[  |N_{A_1}(E_2) \cap N_{A_1}(u) |
 \ge |N_{A_1}(E_2)| + |N_{A_1}(u) |  - |A_1|
 \ge \ell_1 - o(n) \ge 2k+1. \]
  Thus, we can choose a set $E_1$
  consisting of $2k+1$ vertices of
 $N_{A_1}(E_2)\cap N_{A_1}(u)$.
 Note that $|N_{E_1}(u)| =2k+1, |N_{E_2}(u)| \ge 2k+1$ and
 $ G[E_1,E_2]$ is a complete bipartite graph.
By Lemma \ref{no-two-edge} (a), the sets $E_1$,  $E_2$, and the vertex $u$, induce a graph containing a copy of $C_{2k+1}^{\square} $.
This is a contradiction. Thus
there is no edge in $G[A_2]$ incident with a vertex of $N_{A_2}(u)$.
\end{proof}

From Claim \ref{claim2-no-edge}, we obtain
\begin{align}
\notag
e(A_1) &= e(N_{A_1}(u)\cup \{u\}) + e(A_1\setminus (N_{A_1}(u)\cup \{u\}))   \\
\label{eqA1}
& \leq \ell_1+\mbox{ex}(n_1-1- \ell_1,P_4)  \leq \ell_1 + (n_1-1-\ell_1)  \le \mbox{ex}(n_1,P_4),
\end{align}
where the first inequality holds by Claim \ref{claim},
and the others hold by Corollary \ref{cor-6}.
Applying Claim \ref{claim2-no-edge}, we get
\begin{align} \label{eqA2}
e(A_2) = e(A_2 \setminus N_{A_2}(u))
\leq\mbox{ex}(n_2- \ell_2,P_4) \le n_2 -\ell_2.
\end{align}
Since $u$ has $n_2 - \ell_2$ non-neighbors in $A_2$,
we bound the number of edges between $A_1$ and $A_2$,
\begin{equation}  \label{eqA1A2}
e(A_1,A_2) \le n_1n_2 - (n_2 -\ell_2).
\end{equation}
Note that $e(G)=e(A_1,A_2) + e(A_1) + e(A_2)$.
Combining (\ref{eqA1}), (\ref{eqA2}) and (\ref{eqA1A2}), we have
\begin{align} \label{eq-up-bound}
e(G)  \leq n_1n_2+\mbox{ex}(n_1,P_4) \leq\max\left\{n_an_b + \mbox{ex}(n_a,P_4) :n_a+n_b=n\right\} .
\end{align}
Hence, by (\ref{the lower bound}), we have $e(G)=\max\left\{\mbox{ex}(n_a,P_4)+n_an_b:n_a+n_b=n\right\}$.

Furthermore, using Corollary \ref{cor-6}, we obtain
\[ e(G) = \max\limits_{n_a+n_b=n} \left\{n_a(1+n_b)+\frac{1}{2}(j^2-3j):j\in \{0,1,2\},
j\equiv n_a (\mathrm{mod} ~3)\right\}. \]

Finally, we will characterize the extremal graphs.
If the equality holds in (\ref{eq-up-bound}), then from the above proof,
 we must have equalities in (\ref{eqA1}), (\ref{eqA2})
 and (\ref{eqA1A2}). Denote $t_1:=n_1-1- \ell_1$ and $t_2:=n_2-\ell_2$. Then for $i=1,2$, we have
\begin{align} \label{eq-8}
\ex(t_i, P_4)=t_i.
\end{align}
Moreover, attaining equality in (\ref{eqA1A2}) yields that each vertex of $A_1\backslash\{u\}$
 is adjacent to each vertex of $A_2$.
 By Corollary \ref{cor-6}, if $t_i>0$, then $t_i\equiv 0~(\bmod~3)$, and so  $t_i\geq3$.
In addition, Corollary \ref{cor-6} also implies that an extremal graph for $P_4$ on at least 3 vertices contains a copy of $P_3$. If both $t_1$ and $t_2$ are not zero, then by (\ref{eq-8}),
both $G[A_1\backslash\{u\}]$ and $G[A_2]$ contain a copy
of $P_2$. Hence by Lemma \ref{no-two-edge} (b), $G$ contains a copy of $C_{2k+1}^{\square} $, a contradiction.
Thus, at most one of $t_1$ and $t_2$ is not zero.
Firstly,
if $t_1\neq0$ and  $t_2=0$, then $A_2=N_{A_2}(u)$
and $A_2$ is an independent set in $G$.
Recall in Claim \ref{claim} that $G[A_1]$ is
$P_4$-free. Then $G$ is obtained from $K_{n_1,n_2}$
 by embedding an extremal graph for $P_4$ into the part $A_1$.
Secondly, if $t_1=0$ and $t_2\neq0$, then $u$ is adjacent to all vertices
of $A_1 \setminus \{u\}$.
Moving the vertex $u$
to the set $A_2$, and  we
denote $B_1:=A_1\backslash\{u\}$ and $B_2:=A_2\cup\{u\}$.
Recall in Claim \ref{claim2-no-edge} that there is no edge in $G[A_1\setminus \{u\}]$ incident with a vertex of $N_{A_1}(u)=A_1 \setminus \{u\}$, so $B_1$ is an independent set. Moreover, $G[B_1,B_2]$ is a complete bipartite graph and $G[B_2]$
is $P_4$-free, as desired.
Finally, if $t_1=t_2=0$, then
$G[A_1,A_2]$ is a complete bipartite graph, $G[A_1]$ is
a star $K_{1,n_1-1}$, and $G[A_2]$ is an empty graph, as required.
To sum up, every  extremal graph for $C_{2k+1}^{\square} $ is
the join of an $n_a$-vertex extremal graph for $P_4$
and an $n_b$-vertex independent set.
\end{proof}

\section{Some Lemmas for the triangular prism}

\label{sec4}

In this section, we shall consider the triangular prism.
In the case $k=1$ of Theorem
\ref{main},  we will
throw away the condition that requires the order
$n$ being sufficiently large. Using a quite different method,
we will determine all $n$-vertex extremal graphs of
the prism $C_3^{ \square}$ for every integer $n\ge 6$.

\subsection{A result for the power of a path}

The $p$-th power of a path $P_{k}$ on $k$ vertices, denoted by $P_{k}^p$, is the graph obtained from $P_{k}$ by joining each pair of vertices of $P_{k}$ with distance at most $p$; see Figure \ref{P_k}.

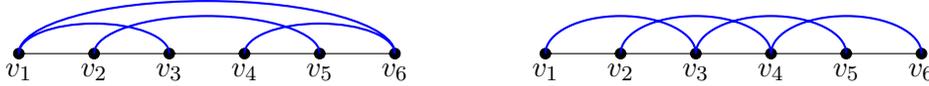
\begin{figure}[H]
\centering
\begin{tikzpicture}

\filldraw (0,0) circle (2pt) node[below]{$v_1$} --
(1,0) circle (2pt) node[below]{$v_2$} --
(2,0) circle (2pt) node[below]{$v_3$} --
(3,0) circle (2pt) node[below]{$v_4$}--
(4,0) circle (2pt)node[below]{$v_5$}--
(5,0) circle (2pt) node[below]{$v_{6}$};

\draw[blue,thick]
 (2,0) arc (0: 180: 1 and 0.4)
 (4,0) arc (0: 180: 1.5 and 0.5)
   (5,0) arc (0: 180: 1 and 0.4)
    (5,0) arc (0: 180: 2.5 and 0.7);

\filldraw (7,0) circle (2pt) node[below]{$v_1$} --
(8,0) circle (2pt) node[below]{$v_2$} --
(9,0) circle (2pt) node[below]{$v_3$} --
(10,0) circle (2pt) node[below]{$v_4$}--
(11,0) circle (2pt)node[below]{$v_5$}--
(12,0) circle (2pt) node[below]{$v_{6}$};

\draw[blue,thick]
 (9,0) arc (0: 180: 1 and 0.5)
 (10,0) arc (0: 180: 1 and 0.5)
  (11,0) arc (0: 180: 1 and 0.5)
   (12,0) arc (0: 180: 1 and 0.5);
\end{tikzpicture}
\caption{The triangular prism $C_{3}^{\square}$  and
the graph $P_{6}^2$, respectively}
\label{P_k}
\end{figure}

In order to study $\ex(n,C_{3}^{\square} )$, we need to use a result from Xiao, Katona, Xiao and Zamora \cite{Xiao} in which they determined the exact value of $\ex(n,P_5^2)$ and $\ex(n,P_6^2)$, respectively.
For the power of a long path, see Yuan \cite{YUAN} for more results.
In addition, the Tur\'{a}n number $\mathrm{ex}(n,P_4^2)$ can be deduced from a general result of
Dirac \cite{DIR}, which states that  if $r\ge 3, n\ge r+1$ and
$G$ is an $n$-vertex graph with more than $e(T_{r-1}(n))$ edges,
then $G$ contains a copy of $K_{r+1}^-$, where $K_{r+1}^-$
is the graph obtained from $K_{r+1}$ by deleting any one edge.
 Setting $r=3$ and observing that $P_4^2=K_4^-$,
 we obtain $\mathrm{ex}(n,P_4^2)=
\lfloor n^2/4\rfloor$ for each integer $n\ge 4$.

Before showing the proofs of Theorems \ref{C3} and \ref{EC3},
we need the following results.

\begin{theorem}[Xiao et al. \cite{Xiao}]  \label{P6}
The maximum number of edges in an $n$-vertex $P^{2}_{6}$-free graph $(n\neq5)$ is
\begin{displaymath} \ex(n,P^{2}_{6})=
\begin{cases}
\left\lfloor\frac{n^{2}}{4}\right\rfloor+\left\lfloor\frac{n-1}{2}\right\rfloor,&
 \text{if}~n\equiv1,2,3 ~(\bmod~6), \\[2mm]
\left\lfloor\frac{n^{2}}{4}\right\rfloor+\left\lceil\frac{n}{2}\right\rceil, &
\text{otherwise.}
\end{cases}
\end{displaymath}
\end{theorem}

\begin{theorem}[Xiao et al. \cite{Xiao}]  \label{EP6}
Let $n\geq 6$. The extremal graphs for $P_6^2$ are the following ones.

When $n\equiv 1 \pmod{6}$ then $F_n^{\lceil {\frac{n}{2}}\rceil, j}$ and $H_n^{\lfloor {\frac{n}{2}}\rfloor}$;

when $n\equiv 2 \pmod{6}$ then $F_n^{{\frac{n}{2}}, j}$ and $F_n^{{\frac{n}{2}}+1, j}$;

when $n\equiv 3 \pmod{6}$ then $F_n^{\lceil {\frac{n}{2}}\rceil, j}$ and $H_n^{\lceil {\frac{n}{2}}\rceil+1}$;

when $n\equiv 0, 4 ,5 \pmod{6}$ then $H_n^{{\frac{n}{2}}}$, $H_n^{{\frac{n}{2}} +1}$ and $H_n^{\lceil {\frac{n}{2}}\rceil}$, respectively, where $j$ can have all the values satisfying the conditions $j\leq i$ and $3 \mid (i-j)$.
\end{theorem}

Comparing the above theorems with
Theorems \ref{C3} and \ref{EC3},
it is surprising that the Tur\'{a}n numbers of graphs $C_{3}^{\square}$
and $P_6^2$ are completely the same,
although their structures are different.
In addition, it is noteworthy that the extremal graphs
for $C_{3}^{\square}$
and $P_6^2$ are also the same except for three small graphs $G_1,G_2$ and $G_3$, depicted  in Figure \ref{678}.

In fact,
it is a key relation between $C_{3}^{\square}$ and $P_6^2$
 that $\mathcal{M}(C_3^{ \square}) =
\mathcal{M}(P_6^2)= \{P_4\}$, where $\mathcal{M}(L)$
is the decomposition family of $L$ stated in Definition \ref{def-8}.
To a certain extent,
the above phenomenon
reveals that the decomposition family of a graph could
determine the Tur\'{a}n number of a graph and
also dominate the structures of the extremal graphs.

\subsection{The $C_{3}^{\square}$-free graphs containing $P_6^2$ as a subgraph}

Before showing the proof of Theorems \ref{C3} and \ref{EC3},
we next present some useful lemmas
for a special class of $C_{3}^{\square}$-free graphs which contain
the $2$-power graph $P_6^2$ as a subgraph.

\medskip

\begin{lemma}\label{n=6}
Let $G$ be a  $C_3^{\square}$-free graph on $6$ vertices.
If $G$ contains a copy of $P_6^2$, then
$ e(G)\leq12$,
 with equality holding if and only if $G=G_1$.
\end{lemma}

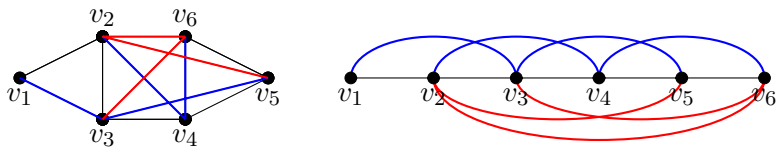
\begin{figure}[H]
\centering
\begin{tikzpicture}[scale=1.1]
\filldraw (0,-0.5) circle (2pt) -- (-1,0) circle (2pt) --
(0,0.5) circle (2pt) --
(2,0) circle (2pt) --
(1,0.5) circle (2pt) --
(1,-0.5) circle (2pt) --
(0,-0.5) circle (2pt) --
 (0,0.5) circle (2pt) --
 (1,0.5) circle (2pt) --
  (0,-0.5) circle (2pt) --
  (2,0) circle (2pt) --
  (1,-0.5) circle (2pt) --
  (0,0.5) circle (2pt);

\filldraw (-1,0) circle (2pt) node[below]{$v_1$} --
(0,0.5) circle (2pt) node[above]{$v_2$} --
(0,-0.5) circle (2pt) node[below]{$v_3$} --
(1,-0.5) circle (2pt) node[below]{$v_4$}--
(1,0.5) circle (2pt)node[above]{$v_6$}--
(2,0) circle (2pt) node[below]{$v_{5}$};

\draw[blue,thick]
 (-1,0) -- (0,-0.5)
 (0,0.5) -- (1,-0.5)
 (0,-0.5) -- (2,0)
 (1,-0.5) -- (1,0.5);

    \draw[red,thick]
    (0,0.5)--(2,0)
    (0,0.5)--(1,0.5)
    (0,-0.5)--(1,0.5);


\draw[blue,thick]
 (5,0) arc (0: 180: 1 and 0.5)
 (6,0) arc (0: 180: 1 and 0.5)
  (7,0) arc (0: 180: 1 and 0.5)
   (8,0) arc (0: 180: 1 and 0.5);

   \draw[red,thick]
 (4,0) arc (180: 360: 1.5 and 0.5)
  (5,0) arc (180: 360: 1.5 and 0.5)
   (4,0) arc (180: 360: 2 and 0.75);

\filldraw (3,0) circle (2pt) node[below]{$v_1$} --
(4,0) circle (2pt) node[below]{$v_2$} --
(5,0) circle (2pt) node[below]{$v_3$} --
(6,0) circle (2pt) node[below]{$v_4$}--
(7,0) circle (2pt)node[below]{$v_5$}--
(8,0) circle (2pt) node[below]{$v_{6}$};
\end{tikzpicture}
\caption{Two drawings of the graph $G_1$}
\label{G1}
\end{figure}

\begin{proof}
The proof is a standard argument on case-analysis.
Since $G$ has $6$ vertices and $P_6^2$ is contained in $G$,
we may assume that $V(G)=\{v_1,v_2,\ldots ,v_6\}$,
where $v_1v_3, v_2v_4,v_3v_5$ and $v_4v_6$ are edges of $G$. Clearly, we have $9\le e(G)\le 15$.
To prove $e(G)\le 12$, it is equivalent to show that there missed at least three edges
in $G$ whenever $K_6$ is $C_{3}^{\square}$-free.

First of all, we claim that $v_1v_6$ can not be an edge of $G$.
Otherwise, if $v_1v_6\in E(G)$, then
two triangles $v_1v_2v_3$ and $v_6v_4v_5$, together with
three edges $v_1v_6, v_2v_4,v_3v_5$ form a copy of $C_{3}^{\square} $, a contradiction.
Thus, we get $2\le d_G(v_1)\le 4$.

{\bf Case 1.} $d_G(v_1)=2$.
More precisely, we have
$v_1v_4,v_1v_5,v_1v_6\notin E(G)$, and $G$ misses
at least three edges of $K_6$.
 Clearly, we have $e(G) \le 12$.
If $e(G)=12$,
then $v_2v_5,v_2v_6$ and $v_3v_6$ are edges of $G$.
In this case, we get $G=G_1$, as required.

{\bf Case 2.}  $d_G(v_1)=3 $.  In other words,
we have either $v_1v_4\in E(G)$ or $v_1v_5 \in E(G)$.
In the former case, if $v_1v_4\in E(G)$, then $v_2v_6 \notin E(G)$.
Otherwise, if $v_2v_6\in E(G)$, then
two triangles $v_1v_2v_3$ and $v_4v_6v_5$, together with
edges $v_1v_4,v_2v_6,v_3v_5$ form a copy of $C_{3}^{\square} $, a contradiction.
Hence, we conclude that $v_1v_6, v_1v_5$ and $v_2v_6$
are not edges of $G$, which leads to $e(G) \le 12$.
If $e(G)=12$, then $v_1v_4,v_2v_5,v_3v_6$ are edges of $G$, which together with two triangles $v_1v_2v_3$ and $v_4v_5v_6$ form a copy of $C_{3}^{\square} $, a contradiction. Hence $e(G)\le 11$.

In the later case, that is, $v_1v_5\in E(G)$ and
$v_1v_4, v_1v_6 \notin E(G)$,
we can get $v_2v_6\notin E(G)$
and $v_3v_6\notin E(G)$.
Otherwise, if $v_2v_6\in E(G)$, then
triangles $v_1v_2v_3$ and $v_5v_6v_4$, with edges
$v_1v_5,v_2v_6,v_3v_4$ form a copy of $C_{3}^{\square} $, a contradiction. If $v_3v_6\in E(G)$,
then triangles $v_1v_2v_3$ and $v_5v_4v_6$, with edges
$v_1v_5, v_2v_4, v_3v_6$ form a copy of $C_{3}^{\square} $, a contradiction. Therefore, we have $e(G)\le 11$.

{\bf Case 3.} $d_G(v_1)=4$.
In this case, we have $v_1v_4,v_1v_5\in E(G)$.  Similar with the later case of Case 2,
we get $v_2v_6\notin E(G)$ and $v_3v_6\notin E(G)$ as well.
Recall that $v_1v_6\notin E(G)$. Consequently, we obtain $e(G)\le 12$.  If $e(G)=12$, then $v_2v_5\in E(G)$.
 By reversing the order of $v_1v_2\cdots v_6$,
it is not hard to see that $G$ is isomorphic to $G_1$,
as required.
\end{proof}

In the following, we always assume that
\begin{equation}  \label{eq-P6}
P^{2}_{6}\subseteq G.
\end{equation}
Assume that the subgraph $P_6^2$ is obtained from $P_6=v_1v_2\cdots v_6$ by joining each pair of vertices of $V(P_6)$ with distance at most $2$.
It can be checked that each vertex $v\in V(G)\setminus V(P_6^2)$ can be adjacent to at most 4 vertices of the copy of $P_6^2$.
Otherwise, if there is a vertex with $5$ neighbors in $P_6^2$,
then we can find a copy of $C_{3}^{\square} $ in $G$, a contradiction.

The following claim holds immediately under the condition (\ref{eq-P6}).

\medskip

\begin{claim} \label{cla-4}
 If $v\in V(G)\setminus V(P_6^2)$ is adjacent to exactly 4 vertices of $P_6^2$, then there are exactly 4 adjacency relations (up to graph isomorphism) between $v$ and $P_6^2$; see Figure \ref{4case}.
 \end{claim}

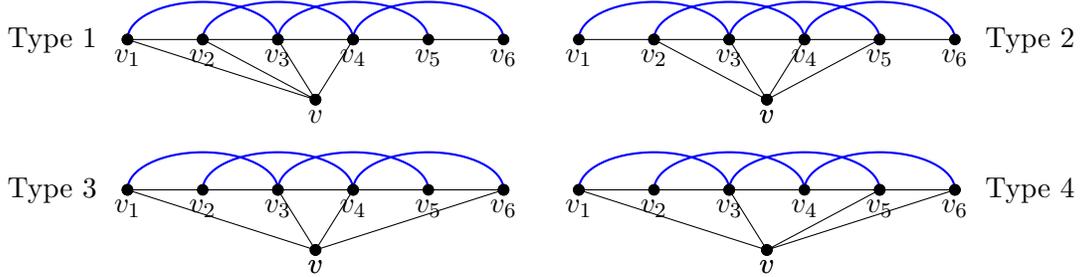
\begin{figure}[H]
\centering
\begin{tikzpicture}

\draw[blue,thick]
 (2,0) arc (0: 180: 1 and 0.5)
  (3,0) arc (0: 180: 1 and 0.5)
   (4,0) arc (0: 180: 1 and 0.5)
    (5,0) arc (0: 180: 1 and 0.5);

\draw (-1,0) node{{Type 1}}
(12,0) node{Type 2}
 (-1,-2) node{Type 3}
 (12,-2) node{Type 4} ;
\filldraw (0,0) circle (2pt) node[below]{$v_1$} -- (1,0) circle (2pt) node[below]{$v_2$} -- (2,0) circle (2pt) node[below]{$v_3$} -- (3,0) circle (2pt) node[below]{$v_4$}--(4,0) circle (2pt)node[below]{$v_5$} -- (5,0) circle (2pt) node[below]{$v_6$};

\filldraw (0,0) circle (2pt) -- (2.5,-0.8) circle (2pt) node[below]{$v$} -- (1,0) circle (2pt);

\filldraw (2,0) circle (2pt) -- (2.5,-0.8) circle (2pt)  -- (3,0) circle (2pt);

\draw[blue,thick]
 (8,0) arc (0: 180: 1 and 0.5)
  (9,0) arc (0: 180: 1 and 0.5)
   (10,0) arc (0: 180: 1 and 0.5)
    (11,0) arc (0: 180: 1 and 0.5);

\filldraw (6,0) circle (2pt) node[below]{$v_1$} -- (7,0) circle (2pt) node[below]{$v_2$} -- (8,0) circle (2pt) node[below]{$v_3$} -- (9,0) circle (2pt) node[below]{$v_4$}--(10,0) circle (2pt)node[below]{$v_5$} -- (11,0) circle (2pt) node[below]{$v_6$};

\filldraw (7,0) circle (2pt) -- (8.5,-0.8) circle (2pt) node[below]{$v$} -- (8,0) circle (2pt);

\filldraw (9,0) circle (2pt) -- (8.5,-0.8) circle (2pt) node[below]{$v$} -- (10,0) circle (2pt);

\draw[blue,thick]
 (2,-2) arc (0: 180: 1 and 0.5)
  (3,-2) arc (0: 180: 1 and 0.5)
   (4,-2) arc (0: 180: 1 and 0.5)
    (5,-2) arc (0: 180: 1 and 0.5);

\filldraw (0,-2) circle (2pt) node[below]{$v_1$} -- (1,-2) circle (2pt) node[below]{$v_2$} -- (2,-2) circle (2pt) node[below]{$v_3$} -- (3,-2) circle (2pt) node[below]{$v_4$}--(4,-2) circle (2pt)node[below]{$v_5$} -- (5,-2) circle (2pt) node[below]{$v_6$};

\filldraw (0,-2) circle (2pt) -- (2.5,-2.8) circle (2pt) node[below]{$v$} -- (2,-2) circle (2pt);

\filldraw (3,-2) circle (2pt) -- (2.5,-2.8) circle (2pt) node[below]{$v$} -- (5,-2) circle (2pt);

\draw[blue,thick]
 (8,-2) arc (0: 180: 1 and 0.5)
  (9,-2) arc (0: 180: 1 and 0.5)
   (10,-2) arc (0: 180: 1 and 0.5)
    (11,-2) arc (0: 180: 1 and 0.5);

\filldraw (6,-2) circle (2pt) node[below]{$v_1$} -- (7,-2) circle (2pt) node[below]{$v_2$} -- (8,-2) circle (2pt) node[below]{$v_3$} -- (9,-2) circle (2pt) node[below]{$v_4$}--(10,-2) circle (2pt)node[below]{$v_5$} -- (11,-2) circle (2pt) node[below]{$v_6$};

\filldraw (6,-2) circle (2pt) -- (8.5,-2.8) circle (2pt) node[below]{$v$} -- (8,-2) circle (2pt);

\filldraw (10,-2) circle (2pt) -- (8.5,-2.8) circle (2pt) node[below]{$v$} -- (11,-2) circle (2pt);
\end{tikzpicture}
\caption{Four types
for the vertex $v$ satisfying $|N(v) \cap V(P^{2}_{6})|=4$}
\label{4case}
\end{figure}

Furthermore,  the following lemma can be verified case by case.

\begin{lemma}\label{claim1}
There are at most $2$ vertices of $V(G)\setminus V(P^{2}_{6})$ adjacent to
exactly $4$ vertices of $V(P^{2}_{6})$. Moreover, if $x,y\in V(G)\setminus V(P_6^2)$ are  such two vertices, then
 there are only $2$ adjacency relations  between $\{x,y\}$ and $P_6^2$, and moreover $xy\notin E(G)$; see Figure \ref{2case}.
\end{lemma}

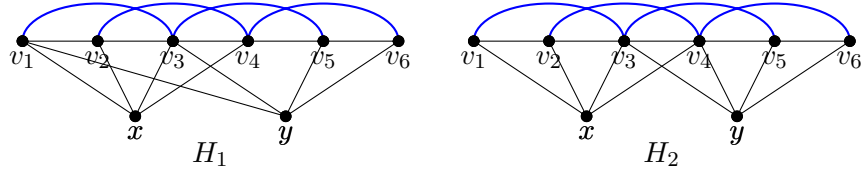
\begin{figure}[H]
\centering
\begin{tikzpicture}

\draw[blue,thick]
 (2,0) arc (0: 180: 1 and 0.5)
  (3,0) arc (0: 180: 1 and 0.5)
   (4,0) arc (0: 180: 1 and 0.5)
    (5,0) arc (0: 180: 1 and 0.5);

\filldraw (0,0) circle (2pt) node[below]{$v_1$} -- (1,0) circle (2pt) node[below]{$v_2$} -- (2,0) circle (2pt) node[below]{$v_3$} -- (3,0) circle (2pt) node[below]{$v_4$}--(4,0) circle (2pt)node[below]{$v_5$} -- (5,0) circle (2pt) node[below]{$v_6$};

\filldraw (0,0) circle (2pt) -- (1.5,-1) circle (2pt) node[below]{$x$} -- (1,0) circle (2pt);

\filldraw (2,0) circle (2pt) -- (1.5,-1) circle (2pt) node[below]{$x$} -- (3,0) circle (2pt);

\filldraw (0,0) circle (2pt) -- (3.5,-1) circle (2pt) node[below]{$y$} -- (2,0) circle (2pt);

\filldraw (4,0) circle (2pt) -- (3.5,-1) circle (2pt) node[below]{$y$} -- (5,0) circle (2pt);
\draw (2.5,-1.5) node{$H_1$};

\draw[blue,thick]
 (8,0) arc (0: 180: 1 and 0.5)
  (9,0) arc (0: 180: 1 and 0.5)
   (10,0) arc (0: 180: 1 and 0.5)
    (11,0) arc (0: 180: 1 and 0.5);

\filldraw (6,0) circle (2pt) node[below]{$v_1$} -- (7,0) circle (2pt) node[below]{$v_2$} -- (8,0) circle (2pt) node[below]{$v_3$} -- (9,0) circle (2pt) node[below]{$v_4$}--(10,0) circle (2pt)node[below]{$v_5$} -- (11,0) circle (2pt) node[below]{$v_6$};

\filldraw (6,0) circle (2pt) -- (7.5,-1) circle (2pt) node[below]{$x$} -- (7,0) circle (2pt);

\filldraw (8,0) circle (2pt) -- (7.5,-1) circle (2pt) node[below]{$x$} -- (9,0) circle (2pt);

\filldraw (8,0) circle (2pt) -- (9.5,-1) circle (2pt) node[below]{$y$} -- (9,0) circle (2pt);

\filldraw (10,0) circle (2pt) -- (9.5,-1) circle (2pt) node[below]{$y$} -- (11,0) circle (2pt);
\draw (8.5,-1.5) node{$H_2$};
\end{tikzpicture}
\caption{The $8$-vertex graphs $H_1$ and $H_2$, respectively.}
\label{2case}
\end{figure}

\begin{proof}
First of all, we show that if $x,y$ are two vertices
in the same type, then $G$ contains a copy of $C_{3}^{\square}$, except that they are of Type 1,
drawing as $H_2$ in Figure \ref{2case}.
Indeed, if $x,y$ are of  Type 1, then
either $N_{P_6^2}(x)=N_{P_6^2}(y)=\{v_1,v_2,v_3,v_4\}$,
or $N_{P_6^2}(x)=\{v_1,v_2,v_3,v_4\}$ and
$N_{P_6^2}(y)=\{v_6,v_5,v_4,v_3\}$.
For the former case,
the triangles $v_1v_3x$ and $v_2v_4y$, together with
three edges $v_1y, v_3v_4, xv_2$, form a copy of
$C_{3}^{\square}$, a contradiction; For the latter case, it is possible, and  stated as  $H_2$ in Figure \ref{2case}.
If $x,y$ are of Type 2, then
the symmetry gives $N_{P_6^2}(x)=N_{P_6^2}(y)=\{v_2,v_3,v_4, v_5\}$. Then the triangles
$v_2v_3x$ and $v_4v_5y$, with the edges
$v_2y, v_3v_5, xv_4$ form a copy of $C_{3}^{\square}$.
If $x,y$ are of Type 3, then
$N_{P_6^2}(x)=N_{P_6^2}(y)=\{v_1,v_3,v_4,v_6\}$.
Thus, the triangles $v_1v_3x$ and $v_4v_6y$
form a copy of $C_{3}^{\square}$ since
$v_1y, v_3v_4, xv_6$ are edges of $G$.
If $x,y$ are of Type 4, then
either $N_{P_6^2}(x)=N_{P_6^2}(y)=\{v_1,v_3,v_5,v_6\}$,
or $N_{P_6^2}(x)=\{v_1,v_3,v_5,v_6\}$
and $N_{P_6^2}(y)=\{v_6,v_4,v_2,v_1\}$.
For the former case, we can see that
triangles $v_1v_3x$ and $v_5v_6y$
form a copy of $C_{3}^{\square}$ since
$v_1y, xv_6, v_3v_5 $ are edges in $G$.
For the latter case, it can be seen that
$G$ contains a copy of $C_{3}^{\square}$
as witnessed by triangles $v_1v_3x$ and
$v_6v_4y$ with three edges $v_1y, xv_6$ and $v_3v_4$.

Secondly, one consider the case that
 $x,y$ are vertices with different types.
More precisely, if $x$ is of Type 1 and $y$ is of Type 2,
then $G$ has a copy of $C_{3}^{\square}$
consisting of the triangles $xv_1v_2, yv_3v_4$
and the edges $v_1v_3, v_2y, xv_4$.
If $x$ is of Type 1 and $y$ is of Type 3,
then the triangles $xv_1v_2, yv_3v_4$
and the edges $v_1y, v_2v_3, xv_4$
form a copy of $C_{3}^{\square}$ in $G$.
If $x$ is of Type 1 and $y$ is of Type 4,
then by symmetry,
we get either $N_{P_6^2}(x)=\{v_1,v_2,v_3,v_4\}$ and
$N_{P_6^2}(y)=\{v_1,v_3,v_5,v_6\}$, or
$N_{P_6^2}(x)=\{v_6,v_5,v_4,v_3\}$
and $N_{P_6^2}(y)=\{v_1,v_3,v_5,v_6\}$.
To avoid a copy of $C_{3}^{\square}$ in $G$,
the former case is possible, which is stated as $H_1$ in Figure \ref{2case}. For the latter case,
one can find a copy of $C_{3}^{\square}$ on
the triangles $xv_4v_6, y v_3v_5$ with the edges
$xv_3, v_4v_5, v_6y$.
If $x$ is of Type 2 and $y$ is of Type 3,
then $G$ contains a copy of $C_{3}^{\square}$,
which appears on the triangles $xv_2v_4, yv_1v_3$
and the edges $xv_3,v_2v_1, v_4y$.
If $x$ is of Type 2 and $y$ is of Type 4,
then $G$ has a copy of $C_{3}^{\square}$
by combining the triangles $xv_3v_4, yv_5v_6$
with the edges $v_3y, v_4v_6, xv_5$.
If $x$ is of Type 3 and $y$ is of Type 4, then
one can find a copy of $C_{3}^{\square}$
on the triangles $xv_1v_3, yv_5v_6$ with the edges
$xv_6, v_1y, v_3v_5$.

Moreover, it can be checked further that $V(G)\setminus V(P_6^2)$
does not contain  three vertices
which are adjacent to $4$ vertices of $P_6^2$.
Otherwise,  if such three vertices exist, i.e.,
we add one more vertex with degree $4$ on $P_6^2$
to the graph $H_1$ or $H_2$,
then one can check that
there is a copy of $C_{3}^{\square}$ in $G$.
Finally, if $xy$ is an edge of $G$, then either $H_1+xy$ or $H_2 +xy$
can have a copy of $C_{3}^{\square}$. Indeed,
in both cases,
the copy will appear on the triangles $v_2v_4x$
and $v_3v_5y$, with the edges $v_2v_3,v_4v_5,xy$.
Thus, we get $xy\notin E(G)$.
\end{proof}

\begin{lemma} \label{lem-H1}
If $G$ contains  $H_1$  as a subgraph,
then $e(G[V(P^{2}_{6})])\leq 10$,
with equality if and only if $v_3v_6$ is an edge in $G$.
\end{lemma}

\begin{proof}
Observe that $e(P_6^2)=9$.
In other words, we need to show that
one can join at most one non-edge of $P_6^2$
to avoid the copy of $C_3^{ \square}$ in $G$.
Note that
\begin{equation} \label{non-edge}
 E(K_6) \setminus E(P_6^2)= \{v_1v_4, v_1v_5, v_1v_6, v_2v_5,  v_2v_6, v_3v_6\}.
 \end{equation}
If $v_1v_4\in E(G[V(P^{2}_{6})])$, then
$G$ contains two triangles $v_1v_4x$ and $v_3v_5y$,
which together with the matching $\{v_1y, v_4v_5,xv_3\}$,
form a copy of $C_{3}^{\square}$, a contradiction.
If $v_1v_5\in E(G[V(P^{2}_{6})])$, then
the triangles $v_1v_2x$ and $v_3v_4v_5$, with
the matching $\{v_1v_5, v_2v_3, xv_4\}$, form a copy of
$C_{3}^{\square}$, a contradiction.
If $v_1v_6\in E(G[V(P^{2}_{6})])$, then
there are two triangles $v_1v_3x$ and $v_4v_5v_6$,
combining with the matching $\{v_1v_6, v_3v_5, xv_4\}$,
we find a copy of $C_{3}^{\square}$, a contradiction.
In fact, whenever $v_1v_6\in E(G)$,
one can find a copy of $C_{3}^{\square}$ without the use of
the vertex $x$, which was showed in the second paragraph of
the proof of Lemma \ref{n=6}.
If $v_2v_5\in E(G[V(P^{2}_{6})])$,
then triangles $v_1v_2x$ and $v_3v_4v_5$,
with the matching $\{v_1v_3, v_2v_5, xv_4\}$, form
 a copy of $C_{3}^{\square}$, a contradiction.
 If $v_2v_6\in E(G[V(P^{2}_{6})])$,
then $G$ contains a copy of $C_{3}^{\square}$, which consists of
two triangles $v_2v_3x$ and $v_4v_5v_6$,
with the matching $\{v_2v_6, v_3v_5, xv_4\}$.
Thus, the only possible case is $v_3v_6 \in E(G[V(P^{2}_{6})])$,
and then it follows that $e(G[V(P^{2}_{6})])\le 10$.
\end{proof}

Similarly, we can prove the following lemma.

\medskip

\begin{lemma} \label{lem-H2}
If $G$ contains $H_2$ as a subgraph,
then $e(G[V(P^{2}_{6})])\leq 11$,
and equality holds if and only if both $v_1v_4$ and $v_3v_6$
are edges in $G$.
\end{lemma}

\begin{proof}
We need to check each non-edge of $E(G[V(P^{2}_{6})])$ stated
 in (\ref{non-edge}).
If $v_1v_5\in E(G[V(P^{2}_{6})])$, then
the triangles $v_1v_2x$ and $v_3v_4v_5$, with
the matching $\{v_1v_5, v_2v_4, xv_3\}$, form a copy of
$C_{3}^{\square}$, a contradiction.
If $v_1v_6\in E(G[V(P^{2}_{6})])$, then
the same argument in the proof of Lemma \ref{lem-H1} can yield a copy of  $C_{3}^{\square}$, a contradiction.
In addition, combining the triangles $v_1v_3x, v_4v_6y$ with
 the matching $\{v_1v_6, v_3y, xv_4\}$,
we find another copy of $C_{3}^{\square}$.
If $v_2v_5\in E(G[V(P^{2}_{6})])$, then
the proof of Lemma \ref{lem-H1} yields a copy of  $C_{3}^{\square}$, a contradiction. Additionally,
there is another copy of $C_{3}^{\square}$ on two triangles $v_2v_3x$ and $v_4v_5y$,
joined by the matching $\{v_2v_5, v_3y, xv_4\}$.
Finally, if $v_2v_6\in E(G[V(P^{2}_{6})])$, then
there is a copy of $C_{3}^{\square}$ in $G$, which
consists of triangles $v_2v_3x, v_4v_6y$
and the matching $\{v_2v_6, v_3y, xv_4\}$.
Therefore, it is possible that
$v_1v_4, v_3v_6 \in E(G[V(P^{2}_{6})])$, and only in such a case, we have $e(G[V(P^{2}_{6})]) = 11$.
\end{proof}

\section{Proofs of Theorems \ref{C3} and \ref{EC3}}

\label{sec5}

In this section, we shall prove Theorems \ref{C3} and \ref{EC3} as follows.

\begin{proof}[{\bf Proof of Theorem \ref{C3}}]
For notational convenience, we denote
 $$
f (n):=
\begin{cases}
\left\lfloor\frac{n^{2}}{4}\right\rfloor+\left\lfloor\frac{n-1}{2}\right\rfloor,&~\text{if}~n\equiv 1,2,3~(\bmod~6), \\[3mm]
\left\lfloor\frac{n^{2}}{4}\right\rfloor+\left\lceil\frac{n}{2}\right\rceil,&~\text{otherwise}.
\end{cases}
$$
First of all,
when $n\equiv0,4,5~(\bmod~6)$,
we have $\ex(n, C_3^{ \square} )\geq
\lfloor\frac{n^{2}}{4}\rfloor+ \lceil\frac{n}{2} \rceil$, which follows from the constructions of $G_1$, $H^{\frac{n}{2}}_{n}$, $H^{\frac{n}{2}+1}_{n}$ and $H^{\lceil\frac{n}{2}\rceil}_{n}$, respectively.
In addition, when $n\equiv1,2,3~(\bmod~6)$,
the lower bound $\ex(n, C_3^{ \square})\geq
 \lfloor\frac{n^{2}}{4} \rfloor+ \lfloor\frac{n-1}{2} \rfloor$ follows from the definitions of $G_2$, $G_3$ and $F^{\lceil\frac{n}{2}\rceil,j}_{n}$.
Thus, it remains to show the following inequality
\begin{eqnarray}\label{3}
\ex(n,C_3^{ \square} )\leq f (n).
\end{eqnarray}
Our proof of (\ref{3}) is by induction on $n$.
Let $G$ be an $n$-vertex $C_3^{ \square} $-free graph.
We need to show  $e(G)\le f (n)$.
Since our inductive step will go from $n-6$ to $n$, we have to prove the base case in each residue class $\bmod~6$.

\begin{center}
{\bf  \S~Base Case~\S}
\end{center}

When $1\le n\leq4$,
it is clear that $K_n$ is the $C_{3}^{\square}$-free graph with the maximum number of edges, and
one can verify that $e(K_{n})=f (n)$.
Moreover,  when $n=6$, by Theorem \ref{P6}, we have
$e(G)\leq 12$ whenever $G$ is $P_6^2$-free;
by Lemma \ref{n=6},  we get $e(G)\leq 12$ whenever $G$ has a copy of $P_6^2$.
Then we always have $e(G) \le 12=f (6)$, as needed.
We remark here that
when $n=5$, we have
\begin{equation}  \label{eq-n=5}
\mathrm{ex}(5,C_3^{ \square} ) =e(K_{5})=10> 9 =f (5).
\end{equation}
This is a counterexample.
 Thus we require $n\neq 5$.

 Next, we are going to prove the statement for $n=11$.
First of all, if $P^{2}_{6}\nsubseteq G$, then by Theorem \ref{P6},
 the results hold immediately. In what follows, we assume that
$P^{2}_{6}\subseteq G$.
Under this constraint, our aim is to show that
\begin{equation} \label{eq-n=11}
  e(G)< f (11)=36.
  \end{equation}

Firstly, if $V(G)\setminus V(P_6^2)$ forms a copy of $K_5$ in $G$, then each $v_i\in V(P_6^2)$ is adjacent to at most 2 vertices of $K_5$, otherwise, there is a copy of $C_{3}^{\square} $ in $G$.
By Lemma \ref{n=6}, we get $e(G[V(P^{2}_{6})])\leq12$. Therefore, we have $e(G)\leq12+12+10=34<36=f (11)$.

Secondly, if
 $V(G)\setminus V(P_6^2)$ forms a copy of $K_5^-$,
 then we assume that $V(K_5^-)=\{u_1,u_2,u_3,u_4,u_5\}$ and $u_4u_5\notin E(K_5^-)$.
 The following two facts can be checked.
The first fact asserts that for each $i\in \{1,\ldots,6\}$, the vertex $v_i$ is adjacent to at most 3 vertices of $K_5^-$, with equality holding if and only if $N_{K_5^-}(v_i)=\{u_1,u_2,u_3\}$, otherwise, it leads to $C_{3}^{\square} \subseteq G$.
The second fact states that there are at most $7$ edges between $\{v_1,v_{2},v_{3}\}$ and $V(K_5^-)$, otherwise, one can get $C_{3}^{\square} \subseteq G$. The similar result holds
between $\{v_4,v_5,v_6\}$ and $K_5^-$.
 Therefore, the number of edges between $V(P^{2}_{6})$ and $V(K_5^-)$ is at most $14$.
 Recall in Lemma \ref{n=6} that $e(G[V(P^{2}_{6})])\leq12$. Therefore, we get $e(G)\leq12+14+9=35<f  (11)$.

Finally, we will consider the remaining case
\begin{equation*}  \label{eq-le8}
e(V(G)\setminus V(P^{2}_{6})) \leq 8.
\end{equation*}
By Lemma \ref{claim1},
we proceed in
 the following three possible cases.

\textbf{Case 1.}\ \ There are exactly 2 vertices of $V(G)
\setminus V(P^{2}_{6})$, say $x$ and $y$, each of them is adjacent to 4 vertices of $V(P^{2}_{6})$. By  Lemma \ref{claim1}, we obtain that either $H_1 \subseteq G$
or $H_2\subseteq G$.

{\bf Subcase 1.1.} If $H_1 \subseteq G$, then by Lemma \ref{lem-H1},
we know that $e(G[V(P^{2}_{6})])\leq 10$.
Consequently, we have
$e(G)\leq
10+(4+4+3+3+3)+8=35<f  (11)$, as desired.

{\bf Subcase 1.2.} If $H_2 \subseteq G$, then by Lemma \ref{lem-H2},
we have $e(G[V(P^{2}_{6})])\leq 11$,
equality holds if and only if both $v_1v_4$ and $v_3v_6$
are edges of $G$.
If there is a vertex of
$V(G)\setminus (V(P^{2}_{6}) \cup \{x,y\})$ adjacent to at most 2 vertices of $V(P^{2}_{6})$, then $e(G)\leq
11+(4+4+2+3+3)+8=35<f (11)$, as desired.
 If every vertex of $V(G)\setminus (V(P^{2}_{6}) \cup \{x,y\})$ is adjacent to 3 vertices of $V(P^{2}_{6})$, then
 it can be seen further that $e(G[V(P^{2}_{6})])\leq 10$.
 Otherwise, if $e(G[V(P^{2}_{6})]) =11$,
 then Lemma \ref{lem-H2} implies $v_1v_4, v_3v_6\in E(G)$.
 Any vertex of $V(G)\setminus (V(P^{2}_{6}) \cup \{x,y\})$
 with degree $3$ on $V(P_6^2)$ can lead to
 a copy of $C_{3}^{\square} $, a contradiction.
 Therefore,
 we have $e(G)\leq 10+(4+4+3\times 3)+8=35<f(11)$.

\textbf{Case 2.}
 There is exactly one vertex  $v\in V(G) \setminus V(P^{2}_{6})$,
 which is adjacent to 4 vertices of $V(P^{2}_{6})$.
By Claim \ref{cla-4},
 $v$ is one of the four types in Figure \ref{4case}.
In either case, by Lemma \ref{n=6},
 we can get $G[V(P^{2}_{6})] \neq G_1$
 and then $e(G[V(P^{2}_{6})])\leq11$. Otherwise,
 if $G[V(P^{2}_{6})] = G_1$, then
 $v_2v_5, v_2v_6 $ and $v_3v_6$ are edges of $G$.
Observe that the vertex set $\{v_2,v_3,\ldots ,v_6\}$
forms a copy of the complete graph $K_5$ in $G$.
Since $v$ has $4$ neighbors on $P_6^2$,
we know that there are at least $3$ neighbors of $v$
which are contained in the set $\{v_2,v_3,\ldots ,v_6\}$.
A trivial fact states that if $v$ has at least $3$ neighbors on $K_5$,
then $\{v\}\cup V(K_5)$ contains a copy of $C_{3}^{\square}$.
Thus $G$ has a copy of $C_{3}^{\square} $.
 Therefore,
 we have $e(G)\leq11+(4+3\times 4)+8=35<f (11)$, as needed.

\textbf{Case 3.}\ \ All vertices of $V(G) \setminus V(P^{2}_{6})$ are adjacent to at most 3 vertices of $V(P^{2}_{6})$.
By Lemma \ref{n=6} again, it follows that  $e(G)\leq12+15+8=35<f (11)$.
In conclusion, we complete the proof of (\ref{eq-n=11}) in the base case $n=11$.

\begin{center}
{\bf \S~Inductive Step~\S}
\end{center}

In what follows, we will show the inductive step.
Suppose that (\ref{3}) holds for all integers $\ell\leq n-1$ and $\ell\neq 5$.
Now we consider the case $n\ge 7$.
Assume that $G$ is a $C_{3}^{\square}$-free graph on
$n$ vertices with maximum number of edges.
We need to show that $e(G) \le f (n)$.
Our inductive step will start from $n-6$ to $n$.
Upon on computations, we can
obtain that
\begin{equation} \label{eqn-6}
f (n) =
f (n-6) + 3n-6.
\end{equation}
First of all,  if $P^{2}_{6}\nsubseteq G$,
then by Theorem \ref{P6}, we get $e(G)\leq f (n)$, as desired.
{\it In what follows,
we assume that $P^{2}_{6}\subseteq G$,
and then we show a $\bm{stronger}$ result}
\begin{equation}  \label{eq-strong}
e(G) \le f (n)-1,
\end{equation}
{\it except for $n=7,8$.}
By Lemma \ref{claim1},
there are at most $2$ vertices of $V(G) \setminus V(P^{2}_{6})$ adjacent to
exactly $4$ vertices of $V(P^{2}_{6})$.
In the sequel, we shall present the proof in three cases.

{\bf Case A.}
If there are exactly $2$ vertices of $V(G) \setminus V(P^{2}_{6})$,
say $x$ and $y$, adjacent to 4 vertices of $P^{2}_{6}$, then
by Lemma \ref{claim1}, we get either $H_1\subseteq G$ or $H_2\subseteq G$.

{\bf Subcase A.1.} Assume that $H_1\subseteq G$.
Then Lemma \ref{lem-H1} gives $e(G[V(P^{2}_{6})])\leq 10$.

  If $e(G[V(P^{2}_{6})])\leq 9$, then
$e(G)\leq9+8+3(n-8)+\ex(n-6,C_{3}^{\square} )
\le 3n-7+f (n-6)
<f (n)$, as needed.
In what follows, we assume that
 $e(G[V(P^{2}_{6})])=10$. Then Lemma \ref{lem-H1} implies that $G[V(P_6^2)] $
 consists of $P_6^2 $ by adding the edge $ v_3v_6$.

 If there exists a vertex of $V(G)\setminus (V(P^{2}_{6})\cup \{x,y\})$ adjacent to at most 2 vertices of $P^{2}_{6}$, then $e(G)\leq10+8+2+3(n-9)+\ex(n-6,C_{3}^{\square} )\le 3n-7+f (n-6)<
 f (n)$,
 where the last inequality holds by (\ref{eqn-6}).

If all vertices of
$V(G) \setminus (V(P^{2}_{6})\cup \{x,y\})$ are adjacent to 3 vertices of $P^{2}_{6}$,
then
 it is easy to check that they have the same neighbors on $P^{2}_{6}$; see $H_3$ in Figure \ref{sameneighbor}.
  Next, we show that $V(G) \setminus V(P_6^2)$
  is an independent set in $G$.
  Indeed, by Lemma \ref{claim1}, we know that
  $xy\notin E(G)$. In Figure \ref{sameneighbor},
  if $xz$ is an edge of $G$, then
  combining with triangles $xv_3z, v_4v_5v_6$
  with edges $v_3v_6, xv_4, zv_5$,
  we see that $G$ has a copy of $C_{3}^{\square}$,
  a contradiction.
  If $yz$ is an edge of $G$, then
  the triangles $v_3yz, v_4v_5v_6$ and edges
  $v_3v_4, yv_5, zv_6$ form a copy of $C_{3}^{\square}$,
  a contradiction. If $zw$ is an edge of $G$, then
  the copy of $C_{3}^{\square}$ will appear on the
  triangles $v_3zw, v_4v_5v_6$ and edges $v_3v_4,
  zv_5, wv_6$, a contradiction.
  To sum up, we get that $G= H_3$;
  see Figure \ref{sameneighbor}.
  Consequently, we have $e(G)= 10 + 8 + 3(n-8)
  <  f (n) $, as desired.

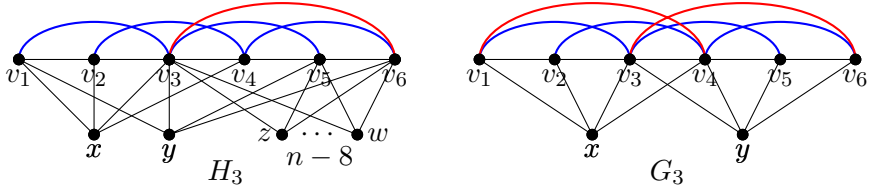
\begin{figure}[H]
\centering
\begin{tikzpicture}

\draw[blue,thick]
 (2,0) arc (0: 180: 1 and 0.5)
  (3,0) arc (0: 180: 1 and 0.5)
   (4,0) arc (0: 180: 1 and 0.5)
    (5,0) arc (0: 180: 1 and 0.5);
 \draw[red,thick]   (5,0) arc (0: 180: 1.5 and 0.75);

\filldraw (0,0) circle (2pt) node[below]{$v_1$} -- (1,0) circle (2pt) node[below]{$v_2$} -- (2,0) circle (2pt) node[below]{$v_3$} -- (3,0) circle (2pt) node[below]{$v_4$}--(4,0) circle (2pt)node[below]{$v_5$} -- (5,0) circle (2pt) node[below]{$v_6$};

\filldraw (0,0) circle (2pt) -- (1,-1) circle (2pt) node[below]{$x$} -- (1,0) circle (2pt);

\filldraw (2,0) circle (2pt) -- (1,-1) circle (2pt) node[below]{$x$} -- (3,0) circle (2pt);

\filldraw (0,0) circle (2pt) -- (2,-1) circle (2pt) node[below]{$y$} -- (2,0) circle (2pt);

\filldraw (4,0) circle (2pt) -- (2,-1) circle (2pt) node[below]{$y$} -- (5,0) circle (2pt);

\filldraw (3.5,-1) circle (2pt) -- (2,0) circle (2pt);

\filldraw (4,0) circle (2pt) -- (3.5,-1) circle (2pt) node[left]{$z$}-- (5,0) circle (2pt);

\filldraw (4.5,-1) circle (2pt) -- (2,0) circle (2pt);

\filldraw (4,0) circle (2pt) -- (4.5,-1) circle (2pt) node[right]{$w$}-- (5,0) circle (2pt);

\draw (2.75,-1.5) node{$H_3$}
(4,-1) node{$\cdots$} (4,-1.3) node{$n-8$};
\end{tikzpicture}
 \quad
\begin{tikzpicture}
\draw[blue,thick]
 (2,0) arc (0: 180: 1 and 0.5)
  (3,0) arc (0: 180: 1 and 0.5)
   (4,0) arc (0: 180: 1 and 0.5)
    (5,0) arc (0: 180: 1 and 0.5);
 \draw[red,thick]
  (3,0) arc (0: 180: 1.5 and 0.75)
  (5,0) arc (0: 180: 1.5 and 0.75);

\filldraw (0,0) circle (2pt) node[below]{$v_1$} -- (1,0) circle (2pt) node[below]{$v_2$} -- (2,0) circle (2pt) node[below]{$v_3$} -- (3,0) circle (2pt) node[below]{$v_4$}--(4,0) circle (2pt)node[below]{$v_5$} -- (5,0) circle (2pt) node[below]{$v_6$};

\filldraw (0,0) circle (2pt) -- (1.5,-1) circle (2pt) node[below]{$x$} -- (1,0) circle (2pt);

\filldraw (2,0) circle (2pt) -- (1.5,-1) circle (2pt) node[below]{$x$} -- (3,0) circle (2pt);

\filldraw (3,0) circle (2pt) -- (3.5,-1) circle (2pt) node[below]{$y$} -- (2,0) circle (2pt);

\filldraw (4,0) circle (2pt) -- (3.5,-1) circle (2pt) node[below]{$y$} -- (5,0) circle (2pt);

\draw (2.5,-1.5) node{$G_3$} ;
\end{tikzpicture}
\caption{The graphs $H_3$ and $G_3$, respectively.}
\label{sameneighbor}
\end{figure}

{\bf Subcase A.2.} Assume that $H_2\subseteq G$.
By Lemma \ref{lem-H2}, we have $e(G[V(P^{2}_{6})]) \le 11$.

 If $e(G[V(P^{2}_{6})])\leq9$, then $e(G)\leq9+8+3(n-8)+\ex(n-6,C_{3}^{\square} )\leq3n-7+f (n-6)<f (n)$, where the last inequality
 follows from (\ref{eqn-6}).

If $10\le e(G[V(P^{2}_{6})])\le 11$, then
by the proof of Lemma \ref{lem-H2}, we know that
$v_1v_4\in E(G)$ or $v_3v_6\in E(G)$.
We need to
prove the following claim.

 \begin{claim}\label{claimH2}
If $H_2$ is a subgraph of $ G$ and  $z\in
V(G) \backslash (V(P_6^2)\cup \{x,y\})$
is a vertex which is adjacent to three vertices of $P_6^2$, then $zx\notin E(G)$ and $zy\notin E(G)$.
\end{claim}

\begin{proof}[Proof of Claim]
The proof is straightforward.
\end{proof}

If $V(G)\setminus (V(P^{2}_{6})\cup \{x,y\})$ has a vertex, say $z$, which is adjacent to at most $2$ vertices of $P^{2}_{6}$, then $e(G)\leq11+8+2+3(n-9)+\ex(n-6,C_{3}^{\square} ) \leq 3n-6+f(n-6)
= f (n)$.
Next, we  show further that
$ e(G)< f (n)$.
Otherwise, if the equality $e(G)= f (n)$ holds,
then the above inequalities must be equalities.
More precisely, we get $e(G[V(P^{2}_{6})]) =11$,
each vertex outside of $V(P_6^2)\cup \{x,y,z\}$
has exactly $3$ neighbors on $P_6^2$,
and the induced subgraph $G[V(G)\setminus V(P_6^2)]$ is an $(n-6)$-vertex
{\it extremal graph} for $C_3^{ \square}$.
By Claim \ref{claimH2}, we see that
each vertex outside of $V(P_6^2)\cup \{x,y,z\}$
can not be adjacent to any vertex of $\{x,y\}$.
Consequently,
in the subgraph $G[V(G)\setminus V(P_6^2)]$,
the vertex $x$ (and $y$) has at most one neighbor,
and the only possible neighbor is $z$.
However, by induction,
we can see that the extremal graphs of order at least 3 for $C_3^{\square}$
do not contain any vertex with degree at most one.
This is a contradiction.
In other words, the subgraph $G[V(G)\setminus V(P_6^2)]$
can not be an extremal graph for $C_{3}^{\square}$.
Then we get $e(G)<f (n)$, as desired in (\ref{eq-strong}).

If all vertices of $V(G)\setminus (V(P^{2}_{6})\cup\{x,y\})$ are adjacent to 3 vertices of $P^{2}_{6}$,
 then by Claim \ref{claimH2}, there is no edge between
 $\{x,y\}$ and $V(G)\setminus (V(P^{2}_{6})\cup\{x,y\})$.
In other words, $x,y$ are two isolated vertices in the induced subgraph
$G[V(G) \setminus V( P^{2}_{6})]$. Therefore, $e(G)\leq 11+8+3(n-8)+ \ex(n-8,C_{3}^{\square} ) \leq 3n-5+f (n-8)<f (n)$,
where the second inequality follows by induction whenever
$n\neq 13$, and the last inequality holds for $n\neq 8$ by
a direct computation.

In particular,
for $n=8$, we get  $V(G)=V(P_6^2) \cup \{x,y\}$
and
\begin{equation} \label{eq-G3}
  e(G)\le 11+8 =19 =f (8) ,
    \end{equation}
where the equality holds if and only if $G$ is obtained from $H_2$ by
adding two edges $v_1v_4$ and $v_3v_6$; see $G_3$ in Figure \ref{sameneighbor}. In the equality case, we get
$G=G_3$ and   $e(G_3) =19 = f (8)$.
In addition, for $n=13$, recall in (\ref{eq-n=5})
that $\ex(5,C_{3}^{\square} ) =e(K_5)=10$, and then
$e(G) \le  11+8 + 3\times 5 +  \ex(5,C_{3}^{\square} )
=44 < 48= f (13) $, as needed.

{\bf Case B.}
If there is exactly one vertex $v\in V(G)\setminus V(P^{2}_{6})$ adjacent to 4 vertices of $P^{2}_{6}$, then
by Claim \ref{cla-4}, we know that
$v$ is one of the four types in Figure \ref{4case}.

Firstly, we claim that $e(G[V(P^{2}_{6})])\leq 11$.
Otherwise, if $e(G[V(P^{2}_{6})])= 12$, then  Lemma \ref{n=6} implies that $v_2v_5,v_2v_6$ and $v_3v_6$
are edges of $G$. Note that the vertex set $\{v_2,v_3,\ldots ,v_6\}$
forms a copy of $K_5$ in $G$.
Since $v$ has $4$ neighbors on $P_6^2$,
it follows that $v$ has at least $3$ neighbors on this copy of $K_5$,
which leads to a copy of $C_{3}^{\square}$ in $ G$,
a contradiction.

{\bf Subcase B.1.}
If $v$ is of Type 2, Type 3 or Type 4 in Figure \ref{4case},
then we can show further that $e(G[V(P^{2}_{6})])\leq10$.
More precisely, if $v$ is of Type 2, then
adding any non-edge of $P_6^2$ to $G$ can make
a copy of $C_{3}^{\square}$, and so
$e(G[V(P^{2}_{6})]) = 9$;
if $v$ is of Type 3, then $v_2v_5$
is  the only possible edge of $G$, and hence
$e(G[V(P^{2}_{6})])\leq 10$;
if $v$ is of Type 4, then $v_3v_6$ is the only possible edge,
and thus $e(G[V(P^{2}_{6})])\leq 10$.
Therefore, we obtain $e(G)\leq10+4+3(n-7)+\ex(n-6,C_{3}^{\square} ) \leq 3n-7+f (n-6)<f (n)$.

 {\bf Subcase B.2.}
 Assume that $v$ is of Type 1 in Figure \ref{4case}.
In this type, we shall show that
\begin{equation} \label{eq-4-non-edges}
 v_1v_5,v_1v_6, v_2v_5,v_2v_6 \notin E(G).
 \end{equation}
Indeed, if $v_1v_5\in E(G)$, then
$G$ has a copy of $C_{3}^{\square}$, which appears on
the triangles $vv_1v_2, v_3v_4v_5$ with the edges
$v_1v_5, v_2v_3, vv_4$.
If $v_1v_6 \in E(G)$, then the triangles
$vv_1v_3, v_4v_5v_6$ and the edges $v_1v_6,v_3v_5,
vv_4$ form a copy of $C_{3}^{\square}$ in $G$.
If $v_2v_5\in E(G)$, then $G$ contains a copy of
$C_{3}^{\square}$ consisting of the triangles
$vv_1v_2, v_3v_4v_5$ and the edges $v_1v_3, v_2v_5,
vv_4$.
If $v_2v_6 \in E(G)$, then $C_{3}^{\square} \subseteq G$
by combining the triangles $vv_2v_3, v_4v_5v_6$
with the edges $v_2v_6, v_3v_5, vv_4$.
This completes the proof of (\ref{eq-4-non-edges}).
Therefore, to avoid a copy of $C_{3}^{\square}$ in  $G$,
the possible edges on $V(P_6^2)$ are
 $v_1v_4$ and $v_3v_6$; see Figure \ref{4case-f}.

\begin{figure}[H]
\centering
\begin{tikzpicture}

\draw[blue,thick]
 (2,0) arc (0: 180: 1 and 0.5)
  (3,0) arc (0: 180: 1 and 0.5)
   (4,0) arc (0: 180: 1 and 0.5)
    (5,0) arc (0: 180: 1 and 0.5);

 \draw[red, thick]
    (3,0) arc (0: 180: 1.5 and 0.75)
    (5,0) arc (0: 180: 1.5 and 0.75);

\filldraw (0,0) circle (2pt) node[below]{$v_1$} -- (1,0) circle (2pt) node[below]{$v_2$} -- (2,0) circle (2pt) node[below]{$v_3$} -- (3,0) circle (2pt) node[below]{$v_4$}--(4,0) circle (2pt)node[below]{$v_5$} -- (5,0) circle (2pt) node[below]{$v_6$};

\filldraw (0,0) circle (2pt) -- (2.5,-0.8) circle (2pt) node[below]{$v$} -- (1,0) circle (2pt);

\filldraw (2,0) circle (2pt) -- (2.5,-0.8) circle (2pt)  -- (3,0) circle (2pt);
\end{tikzpicture}
\caption{$v$ is of Type 1 and $e(G[V(P_6^2)]) =11$}
\label{4case-f}
\end{figure}
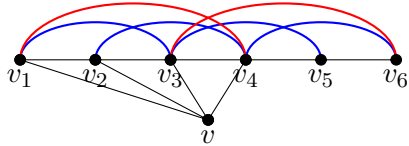

First of all, if $e(G[V(P^{2}_{6})])\leq10$,
then $e(G)\leq10+4+3(n-7)+\ex(n-6,C_{3}^{\square} )
\leq3n-7+f (n-6)<f (n)$, as required.
Next, we assume that $e(G[V(P^{2}_{6})])=11$.
By (\ref{eq-4-non-edges}), we know that
$v_1v_4, v_3v_6\in E(G)$; see Figure \ref{4case-f}.

If  there exist a vertex of
$V(G)\setminus (V(P^{2}_{6})\cup \{v\})$ adjacent to at most $2$ vertices of $P^{2}_{6}$, then
there are at most $4+2+3(n-8)$ edges between
$V(P_6^2)$ and $V(G) \setminus V(P_6^2)$. Thus, we get
$e(G)\leq 11+4+2+3(n-8)+\ex(n-6,C_{3}^{\square} ) \leq
3n-7+f (n-6)<f (n)$.

If all vertices of $V(G)\setminus (V(P^{2}_{6})\cup \{v\})$ are adjacent to $3$ vertices of $P^{2}_{6}$,
then there are exactly $4+3(n-7)$ edges between
$V(P_6^2)$ and $V(G) \setminus V(P_6^2)$.
Since each vertex $x\in V(G)\setminus (V(P^{2}_{6})\cup \{v\})$ has $3$ neighbors on $P_6^3$, we can check that
$vx \notin E(G)$.
Otherwise, one can find a copy of $C_{3}^{\square}$ in $ G$. Therefore, $v$ is an isolated vertex
in the induced subgraph $G[V(G) \setminus V(P_6^2)]$.
Then $e(G)\leq11+4+3(n-7)+\ex(n-7,C_{3}^{\square} ) \leq
3n-6+f (n-7)<f (n)$,
where the second inequality follows by induction whenever  $n\neq 12$, and the last inequality holds for $n\neq 7$
by calculations.

In particular, for $n=7$,
we have $V(G)=V(P_6^2) \cup \{v\}$ and
\begin{equation} \label{eq-G2}
e(G)\le 11+4=15= f (7),
\end{equation}
where equality holds if and only if
$G$ is a graph of Type 1 by adding edges $v_1v_4$ and
$v_3v_6$; see Figure \ref{4case-f}. In the equality case,
we see that $G= G_2$ and $e(G_2)=15=f (7)$.
Moreover, for $n=12$, recall in (\ref{eq-n=5}) that
$\mathrm{ex}(5,C_{3}^{\square})=e(K_5)=10$, then we
have $e(G)\le 11+4+3\times 5+\mathrm{ex}(5,C_{3}^{\square}) = 40 < 42 = f (12)$, as desired in (\ref{eq-strong}).

{\bf Case C.}
Assume that $V(G)\setminus V(P_6^2)$ has no vertex which has $4$ neighbors on $P_6^2$.
By Lemma \ref{n=6}, we have $e(G[V(P_6^2)])
\le 12$. We show the proof in two subcases.

{\bf Subcase C.1.}
There is a vertex of $V(G)\setminus V(P^{2}_{6})$ adjacent to at most 2 vertices of $P^{2}_{6}$. Then the number of edges between
$V(P_6^2)$ and $V(G)\setminus V(P_6^2)$
is at most $ 2+3(n-7)$ since
the other vertices of $V(G)\setminus V(P_6^2)$ adjacent to at most 3 vertices of $P^{2}_{6}$. Thus, we get  $e(G)\leq12+2+3(n-7)+\ex(n-6,C_{3}^{\square} )\leq 3n-7+f (n-6)<f (n)$.

{\bf Subcase C.2.}
All vertices of $V(G)\setminus V(P^{2}_{6})$ are adjacent to exactly 3 vertices of $P^{2}_{6}$.

If  $e(G[V(P^{2}_{6})])\le 11$, then $e(G)\leq11+3(n-6)+\ex(n-6,C_{3}^{\square} )\leq 3n-7+f (n-6)<f (n)$,
where the second inequality holds by induction.

If $e(G[V(P^{2}_{6})])=12$, then
by Lemma \ref{n=6}, we know that
$G[V(P_6^2)]=G_1$, and  so
$v_2v_5,v_2v_6$ and $v_3v_6$ are edges of $G$.
Note that $G[\{v_2,v_3,\ldots ,v_6\}]$ forms a copy of $K_5$;
see Figure \ref{H4}.
We next show that for every $y\in V(G) \setminus V(P_6^2)$,
\begin{equation} \label{eq-same}
 N_{P_6^2}(y)=\{v_1,v_2,v_3\}.
 \end{equation}
Choosing any vertex $x\in V(G)\setminus V(P_6^2)$.
If $N_{P_6^2}(x)\subseteq \{v_2,v_3,\ldots ,v_6\}$,
say $N_{P_6^2}(x)=\{v_3,v_4,v_5\}$ by symmetry,
then $G$ contains a copy of $C_{3}^{\square}$,
which consists of the triangles $xv_3v_4, v_2v_5v_6$
and the edges $xv_5,v_3v_2,v_4v_6$, a contradiction.
Hence, we get $N_{P_6^2}(x)\nsubseteq \{v_2,v_3,\ldots ,v_6\}$.
Without loss of generality, we may assume that
$x\in V(G)\setminus V(P_6^2)$ is a vertex with $N_{P_6^2}(x)=\{v_1,v_2,v_3\}$.
Moreover,
if $V(G)\setminus V(P_6^2)$ contains a vertex $y$ differing from $x$  such that  $N_{P_6^2}(y)\neq N_{P_6^2}(x)$, then
by symmetry, we  assume that either $N_{P_6^2}(y) = \{v_1,v_2,v_4\}$ or $N_{P_6^2}(y)=\{v_1,v_4,v_5\}$.
In the former case,
it is easy to see that the triangles $xv_1v_3, yv_2v_4$
and the edges $xv_2,v_1y, v_3v_4$ form
a copy of $C_{3}^{\square}$.
In fact, one can find another copy without using vertex $x$
by combining two triangles $yv_1v_2, v_3v_4v_5$ with edges $yv_4, v_1v_3, v_2v_5$.
In the latter case, we also see that
$G$ has a copy of $C_{3}^{\square}$ consisting of
the triangles $v_1v_2v_3, y v_4v_5$ and the edges
$v_1y, v_2v_4, v_3v_5$. This is a contradiction.
Consequently, we complete the proof of (\ref{eq-same})
and $H_4\subseteq G$; see Figure \ref{H4}.

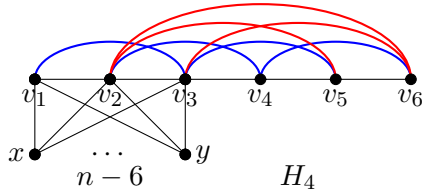
\begin{figure}[H]
\centering
\begin{tikzpicture}

\draw[blue,thick]
 (2,0) arc (0: 180: 1 and 0.5)
  (3,0) arc (0: 180: 1 and 0.5)
   (4,0) arc (0: 180: 1 and 0.5)
    (5,0) arc (0: 180: 1 and 0.5);

    \draw[red,thick]
   (4,0) arc (0: 180: 1.5 and 0.75)
    (5,0) arc (0: 180: 1.5 and 0.75);
        \draw[red,thick]
    (5,0) arc (0: 180: 2 and 1);

\filldraw (0,0) circle (2pt)
node[below]{$v_1$}  -- (1,0) circle (2pt)
node[below]{$v_2$} -- (2,0) circle (2pt)
node[below]{$v_3$} -- (3,0) circle (2pt)
node[below]{$v_4$} --(4,0) circle (2pt)
node[below]{$v_5$} -- (5,0) circle (2pt) node[below]{$v_6$} ;

\filldraw (0,0) circle (2pt) -- (2,-1) circle (2pt) -- (1,0) circle (2pt);

\filldraw (2,0) circle (2pt) -- (2,-1) circle (2pt);

\filldraw (0,0) circle (2pt) -- (0,-1) circle (2pt) -- (1,0) circle (2pt);

\filldraw (2,0) circle (2pt) -- (0,-1) circle (2pt);
\draw (1,-1) node{$\cdots$} (1,-1.3) node{$n-6$}
 (0,-1) node[left]{$x$}  (2,-1) node[right]{$y$}
(3.5,-1.3) node{$H_4$};
\end{tikzpicture}
\caption{The graph $H_4$.}
\label{H4}
\end{figure}

Since $v_1v_2v_3$ is a triangle, to avoid a copy of
$C_3^{ \square} $ in $G$,
we obtain that $G[V(G)\setminus V(P^{2}_{6})]$ is triangle-free. Hence, by Mantel's result (\ref{eq-man}),
we have $e(G)\leq12+3(n-6)+
\lfloor\frac{(n-6)^2}{4} \rfloor=3n-6+ \lfloor\frac{(n-6)^2}{4} \rfloor
<f (n)$, as required.
In a word, we complete the proof of (\ref{eq-strong}).
\end{proof}

In fact, our proof of Theorem \ref{C3}
implies the following stronger result.

\medskip

\begin{theorem} \label{thm-20}
Let $n\ge 9$ be an integer. If $G$ is a graph
with maximum number of edges over all
$C_{3}^{\square}$-free graphs on $n$ vertices, then
 $G$ is an extremal graph for $P_6^2$.
\end{theorem}

\begin{proof}
Recalling the proof of (\ref{eq-strong}), we are done.
\end{proof}

Combining with Theorem \ref{EP6}, we give the
extremal graphs for $C_3^{ \square}$.

\begin{proof}[{\bf Proof of Theorem \ref{EC3}}]
 Let $G$ be an $n$-vertex $C_{3}^{\square} $-free graph
 with maximum number of edges.
 In other words, $e(G)=\mathrm{ex}(n,C_{3}^{\square})=
 f (n)$.
For $n=6$, if $G$ contains
 $P_6^2$ as a subgraph, then  by Lemma \ref{n=6},
 we have $e(G) = 12$ and $G=G_1$;
 if $G$ is $P_6^2$-free, then
by Theorem \ref{EP6}, we get $G=H_6^3$.
For $n=7$, if $G$ contains $P_6^2$ as a subgraph, then
the proof of (\ref{eq-G2}) implies that $G=G_2$;
if $G$ is $P_6^2$-free, then applying Theorem \ref{EP6},
we obtain $G\in \{F_7^{4,1}, F_7^{4,4}, H_7^3\}$.
For $n=8$,
if $G$ has a copy of $P_6^2$, then
the proof of (\ref{eq-G3}) implies that $G=G_3$;
if $G$ is $P_6^2$-free, then Theorem \ref{EP6}
gives that $G\in \{F_8^{4,1},F_8^{4,4},F_8^{5,2}, F_8^{5,5}\}$.

 When $n\geq9$, Theorem \ref{thm-20}
 asserts that $G$ is an extremal graph for $P_6^2$.
 Indeed, if $P_6^2 \subseteq G$, then
 by (\ref{eq-strong}) we know that
 $e(G) \le f (n)-1$, a contradiction.
 Therefore,  using Theorem \ref{EP6}  again,
 we can obtaine all extremal graph for $C_{3}^{\square}$ immediately.
 \end{proof}

 \section{Concluding remarks}
 \label{sec6}

 In Theorem \ref{main},
 we determined the Tur\'{a}n number
 $\mathrm{ex}(n,C_{2k+1}^{ \square})$
 and characterized the corresponding extremal graphs
  for fixed $k\ge 1$ and sufficiently  large $n$.
 Note that $C_{2k}^{\square} $ is a bipartite graph,
 and it is a regular graph with degree $3$.
 For such a spare bipartite graph,
we can obtain that
for every $k\ge 2$,
\begin{equation} \label{eq-C2k}
 \mathrm{ex}(n,C_{2k}^{\square}) =O (n^{5/3}).
 \end{equation}
Indeed, it can be derived
from a result of F\"{u}redi \cite{Furedi91} that
if $H$ is a bipartite graph with maximum degree at most $d$
on one side of the bipartition, then $\mathrm{ex}(n,H) =O(n^{2-1/d})$.
In 2003, Alon, Krivelevich and Sudakov \cite{Alon03}
provided an alternative proof of such a result
as one of the applications of the celebrated dependent random choice method; see, e.g., \cite{FS2011}
for a description of this powerful probabilistic technique.
It seems extremely difficult to determine the  magnitude of
the Tur\'{a}n number of $C_{2k}^{\square}$.
For example, in the case $k=2$,
it reduces to the cube $Q_8$,
Simonovits \cite{ES1970} proved in 1970
that $\mathrm{ex}(n,Q_8)=O (n^{8/5})$. Until now,
it remains unknown whether
$\mathrm{ex}(n,Q_8)= \Theta (n^{8/5})$.
Recently, Conlon and Lee \cite{CL19} conjectured that
if $H$ is a bipartite graph with maximum degree at most $d$
on one side, and $H$ is $K_{d,d}$-free, then
$\mathrm{ex}(n,H) = O(n^{2-1/d -\varepsilon})$
holds for some $\varepsilon = \varepsilon (H)>0$;
see \cite{CL19,Jan2019} for proofs of the case $d=2$,
and \cite{CJL20,ST2020} for more related results.
Observe that $C_{2k}^{\square}$ is $K_{3,3}$-free.
Motivated by the above works, we would like to propose the following conjecture,
which generalized the bound in (\ref{eq-C2k}) exponentially.

 \medskip
  \begin{conjecture}
For every $k\ge 2$, there exists $\varepsilon =\varepsilon (k)>0$
such that
\[ \mathrm{ex}(n,C_{2k}^{\square} ) =
O(n^{\frac{5}{3}- \varepsilon}).\]
 \end{conjecture}

After our paper announced on arXiv, Gao, Janzer, Liu and Xu \cite{GJLX2023}
confirmed the conjecture in a stronger form $\mathrm{ex}(n,C_{2k}^{\square} ) =  \Theta(n^{3/2})$ for every $k\ge 4$.
Extremal results for `twisted' prisms can be found in
 Brada\v{c},  Methuku and Sudakov \cite{BMS2023}.
 Finally, we conclude this paper with an interesting conjecture
proposed by  Brada\v{c}, Janzer, Sudakov and Tomon \cite{BJST2022}.

 \medskip
 \begin{conjecture}[See \cite{BJST2022}]
Let $T, S$ be any two trees with at least one edge.
There exist positive real numbers $c$ and $C$ such that
$  cn^{3/2} \le \mathrm{ex}(n,T \square S) \le Cn^{3/2}$.
 \end{conjecture}

\section*{Declaration of competing interest}
The authors declare that they have no conflicts of interest to this work.

\section*{Data availability}
No data was used for the research described in the article.

\section*{Acknowledgments}
Thanks  go to Dr. Zixiang Xu  for pointing out the reference \cite{JMY2021}.
This work was supported by NSFC (Nos. 12271527, 12071484 and 11931002), Natural Science Foundation of Hunan Province (Nos. 2020JJ4675 and 2021JJ40707), the Fundamental Research Funds for the Central Universities of Central South University (Grant No. 2021zzts0034) and 
 the Postdoctoral Fellowship Program of CPSF (No. GZC20233196).

\end{document}